\documentclass[aop,MSNbibl,citesort,seceqn,dvips]{arximspdf}
\usepackage{wasysym}
\usepackage{stmaryrd}
\usepackage{graphicx}

\doi{10.1214/11-AOP729} 
\volume{41}
\issue{4}
\pubyear{2013}
\firstpage{2990}
\lastpage{3025}

\makeatletter
\newcommand\sC{\mathcal{C}}
\newcommand\sG{\mathcal{G}}
\newcommand\ol{\overline}
\newcommand\oo{\infty}
\newcommand\rad{\operatorname{rad}}
\renewcommand\b{\beta}
\newcommand\g{\gamma}
\renewcommand\a{\alpha}
\newcommand\Om{\Omega}
\newcommand\Ga{\Gamma}
\newcommand\La{\Lambda}
\newcommand\sS{\mathcal{S}}
\newcommand\pc{p_{\mathrm{c}}}
\newcommand\de{\delta}
\newcommand\lra{\leftrightarrow}

\newcommand\xleftrightarrowi{\longleftarrow\!\!\longrightarrow}
\newcommand\xleftrightarrowt{\leftarrow\!\!\rightarrow}
\renewcommand\th{\theta}
\newcommand\bp{\mathbf{p}}
\newcommand\tri{{\triangle}}
\newcommand\hex{{\hexagon}}
\newcommand\Cv{C_{\mathrm{v}}}
\newcommand\Ch{C_{\mathrm{h}}}
\newcommand\ph{p_{\mathrm{h}}}
\newcommand\pv{p_{\mathrm{v}}}
\newcommand\bX{\mathbf{X}}
\newcommand\lest{\le_{\mathrm{st}}}
\newcommand\sE{\mathcal{E}}
\newcommand\sP{\mathcal{P}}
\newcommand\rl{\mathrm{l}}
\newcommand\ul{\underline}
\newcommand\bi{\mathbf{i}}
\newcommand\bj{\mathbf{j}}
\newcommand\what{\widehat}
\newcommand\eps{\varepsilon}
\newcommand\bnu{{\bolds{\nu}}}

\newcommand{\RR}{\mathbb{R}} 
\newcommand{\NN}{\mathbb{N}} 
\newcommand{\ZZ}{\mathbb{Z}} 
\newcommand{\PP}{\mathbb{P}} 
\newcommand{\law}{\mathcal{L}} 
\newcommand{\Dom}{\mathcal{D}} 
\newcommand{\Su}{S^{\Yup}}
\newcommand{\Sd}{S^\Ydown}
\newcommand{\Tu}{T^\vartriangle}
\newcommand{\Td}{T^\triangledown}
\newcommand{\col}{\sC}
\newcommand{\q}{\mathbf{q}}
\newcommand{\p}{\mathbf{p}}
\newcommand{\Lattice}{\mathbb{L}}
\newcommand{\LatticeH}{\mathbb{H}}
\newcommand\Lat{\Lattice}

\renewcommand\r{{\mathbf r}}
\newcommand{\Pqsq}{\mathbb{P}_{\q, \r, p}}
\newcommand{\eqref}[1]{(\ref{#1})}

\newtheorem{thmm}{Theorem}[section]
\newtheorem{lemma}[thmm]{Lemma}
\newtheorem{prop}[thmm]{Proposition}
\newproclaim{definition}[thmm]{Definition}
\newproclaim{obs}[thmm]{Remark}
\newcommand\comp{\circ}

\makeatother

\begin{document}
\begin{frontmatter}

\title{Inhomogeneous bond percolation on square, triangular and hexagonal lattices}
\runtitle{Inhomogeneous bond percolation}

\begin{aug}
\author[A]{\fnms{Geoffrey R.} \snm{Grimmett}\corref{}\ead[label=e1]{g.r.grimmett@statslab.cam.ac.uk}\ead[label=u1,url]{http://www.statslab.cam.ac.uk/\textasciitilde grg/}\thanksref{t1}}
\and
\author[A]{\fnms{Ioan} \snm{Manolescu}\ead[label=e2]{i.manolescu@statslab.cam.ac.uk}\thanksref{t2}}
\thankstext{t1}{Supported in part by the EPSRC under Grant
EP/103372X/1.}
\thankstext{t2}{Supported by the EPSRC and Cambridge University.}
\runauthor{G. R. Grimmett and I. Manolescu}
\affiliation{University of Cambridge}
\address[A]{Statistical Laboratory\\
Centre for Mathematical Sciences\\
University of Cambridge\\
Wilberforce Road\\
Cambridge CB3 0WB\\
United Kingdom\\
\printead{e1}\\
\phantom{E-mail:\ }\printead*{e2}\\
\printead{u1}}
\end{aug}

\received{\smonth{6} \syear{2011}}
\revised{\smonth{10} \syear{2011}}

%
\begin{abstract}
The star--triangle transformation is used to obtain an equivalence
extending over the set of all
(in)homogeneous bond percolation models on the square,
triangular and hexagonal lattices. Among the consequences
are box-crossing (RSW) inequalities for such models with
parameter-values at which
the transformation is valid. This is a step toward
proving the universality and
conformality of these processes. It implies criticality of such values,
thereby providing
a new proof of the critical point of inhomogeneous systems. The proofs extend
to certain isoradial models to which previous methods do not apply.
\end{abstract}

%


\end{frontmatter}
%

\section{Introduction and results}
\subsection{Overview}
Two-dimensional percolation was studied intensively in the early 1980s,
and again in
the decade since 2000. The principal catalyst of the first period was
the rigorous calculation
by Kesten of a certain critical
point (see~\cite{Kesten80}) and of the second the proof by Smirnov of
Cardy's formula (see~\cite{Smirnov}).
The techniques derived during the first period have proved well adapted
to the needs of the second.
For example, the RSW box-crossing lemmas of~\cite{Russo,Seymour-Welsh}
have a key role in
the study of the conformality of critical percolation. Furthermore,
work of Kesten~\cite{Kesten87} on scaling
theory has provided the groundwork for
exact calculations of critical exponents for two-dimensional
percolation (see~\cite{Nolin,Smirnov-Werner},
e.g.).

A great deal of rigorous mathematics exists for critical site
percolation on the triangular
lattice, but surprisingly little for other critical two-dimensional
models. A major new idea is apparently needed
if we are to develop a parallel theory for, say, critical bond
percolation on the square lattice. Another question is
how to extend methods for homogeneous models to inhomogeneous systems.
The purpose of this
article is to explain one part of how this may be done for
(in)homogeneous bond models on square,
triangular and hexagonal lattices.

The Russo--Seymour--Welsh (RSW) theory of box-crossings (see~\cite{GrimmettPercolation}, Section~11.7)
plays a significant part in the theory of critical site percolation on
the triangular lattice---it is used,
for example, in the proof of Cardy's formula, where it implies that
certain crossing probabilities are uniformly H\"older
(see~\cite{GrimmettGraphs,WWparkcity}). RSW theory applies also to
homogeneous bond
percolation on the square lattice. In contrast, it has been an open problem
(see~\cite{JBquestion,BR,Kestenbook})
to derive an RSW
theory for \textit{inhomogeneous} bond percolation (with
edge-probabilities depending on edge-orientation).
The principal methodological advance of the current paper is a
demonstration that
box-crossing inequalities for one critical model may be translated into
inequalities for all other critical
models within the family of (in)homogeneous bond percolation processes
on square, triangular and
hexagonal lattices. Since such inequalities are known for, say,
homogeneous bond percolation on the square lattice, we establish
them thus for all other models of this family.

This progress is achieved by judicious use of the
star--triangle transformation, inspired by the work
\cite{Baxter399} of Baxter and Enting on the Ising model. The star--triangle
transformation and its ramifications (known as the Yang--Baxter
equation) have been at
the heart of many advances in ``exact solutions'' for systems such as the
six-vertex model,
the chiral Potts model, the dimer model and so on (see, e.g.,
\mbox{\cite{Baxterbook,BdTB,Ken02,McCoybook}}).
It turns out that the star--triangle map has a special affinity for
certain physical models
on so-called isoradial graphs---see, for example, the proof of
conformality for the Ising
model by Smirnov~\cite{Chelkak-Smirnov}.
The relationship with isoradial graphs plays a role in the current
work, particularly
in the results for the ``highly inhomogeneous models'' of Theorems
\ref{generalsquareRSW}--\ref{generalsquarecriticality}.

Our basic approach is to use the star--triangle transformation
to transport open paths from one lattice to another. Some complications
arise during transportation, and these may be controlled using probabilistic
estimates. In a second paper~\cite{GM2}, we use these methods
to prove the universality of certain critical exponents, including the
arm exponents, within
the above family of bond percolation models,
under the assumption that they exist for
any one such model.
The current paper has therefore two principal targets:
to prove a theory of box-crossings of
critical inhomogeneous bond percolation models
in two dimensions,
and to develop techniques for the study of universality across families
of such models.

For the percolation models studied in this paper, Theorem \ref
{mainresult} verifies the
assumption of~\cite{Kes86} under
which the incipient infinite cluster has been shown to exist.
Box-crossing inequalities are useful also in proving scale-invariance
for critical percolation in two dimensions (see~\cite{CamNew2,Smirnov}).
Consider, for example, a domain $\Dom$ in the plane (with a
superimposed lattice $\Lattice$ with mesh-size $\delta$)
and four points, $A$, $B$, $C$ and $D$ distributed anti-clockwise along
its boundary.
Consider the limit (as $\delta\to0$) of the probability
that there exists an open path in $\Dom$ joining
the boundary arcs $AB$ and $CD$.
Cardy~\cite{Cardy}
presented a formula for this limit, and
this was proved by Smirnov~\cite{Smirnov}
for the special case of critical site percolation on the triangular lattice.
Corresponding statements are expected to hold for other lattices but no
proofs are yet known.
One may show that, if the underlying measure has the box-crossing
property (see
Definition~\ref{defbxp}),
then such probabilities are bounded uniformly away from 0 and 1 as
$\delta\to0$.

The box-crossing property plays a significant role in the proof of
Cardy's formula,
in which one shows the uniform convergence of a certain triplet of
discretely harmonic functions
to a limiting triplet of harmonic functions.
This is obtained in two steps: first, one proves tightness for the
family of functions,
then one identifies its subsequential limits.
Tightness follows by an application of the Arzel\`a--Ascoli theorem,
whose pre-compactness hypothesis is met by the fact that the discretely harmonic
functions are uniformly H\"older. The proof of this last fact is via
the box-crossing property.
A full proof of Cardy's formula may be found in
\cite{GrimmettGraphs}, Section~5.7 and~\cite{WWparkcity}, Section~2.

The structure of the paper is as follows.
The necessary notation is introduced in Section~\ref{sectnotation},
and our main
results stated in Sections~\ref{sectresults}--\ref{sechim}.
The star--triangle transformation
is discussed in detail in Section~\ref{secttransformation},
with particular attention to transformations of edge-configurations
and open paths. Proofs for inhomogeneous bond percolation on
the square, triangular and hexagonal lattices
are found in Section~\ref{sectmainpf}, and for the highly
inhomogeneous models in Section
\ref{sectgeneralizations}.

\subsection{Notation}\label{sectnotation}
The lattices under study are the square, triangular
and hexagonal (or honeycomb) lattices illustrated in Figure \ref
{figlatticeexamples}.
The hypotheses and conclusions of this paper may often be expressed in
terms of
their graph-theoretic properties. Nevertheless,
we choose here to make use of certain planar embeddings
of these lattices.

An embedding of a planar graph in $\RR^2$ gives rise to a so-called
\textit{dual graph}. We shall
make frequent use of duality, for a short account of which the reader
is referred to~\cite{GrimmettPercolation}, Section~11.2.

\begin{figure}

\includegraphics{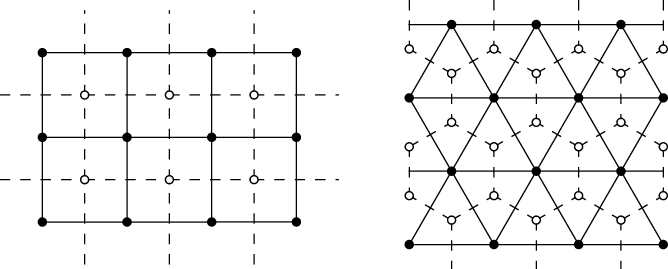}

\caption{The square lattice and its dual square lattice. The
triangular lattice and its dual hexagonal lattice.}
\label{figlatticeexamples}
\end{figure}

With each of the lattices of Figure~\ref{figlatticeexamples}
we may associate a bond percolation model as follows.
Let $G = (V,E)$ be a countable connected graph. A \textit{configuration} on
$G$ is an element $\omega=(\omega_e\dvtx  e \in E)$ of the set $\Om=\{
0,1\}^E$.
An edge with endpoints $u$, $v$ is denoted $uv$.
The edge $e$ is called \textit{open}, or $\omega$-\textit{open},
in $\omega\in\Om$ (resp., \textit{closed}) if $\omega_e=1$
(resp., $\omega_e = 0$).

For $\omega\in\Om$ and $A,B \in V$, we say $A$ \textit{is connected
to} $B$ (in $\omega$),
written $A \leftrightarrow B$ (or $A\displaystyle\mathop{\xleftrightarrowt}^{G,\omega} B$),
if $G$ contains a path of open edges from $A$ to $B$.
An \textit{open cluster} of $\omega$ is a maximal set of
pairwise-connected vertices.
We write $A \lra\oo$ if $A$ is the endpoint of an infinite open
self-avoiding path.

The simplest bond percolation model on $G$ is that associated with the
product measure
$\PP_p$ on $\Om$ with given intensity $p\in[0,1]$. Let $0$ be a
designated vertex
of~$V$ called the \textit{origin}, and define the
\textit{percolation probability}
\[
\th(p) = \PP_p(0 \lra\infty).
\]
The \textit{critical probability} is given by
\[
\pc(G) = \sup\{p\dvtx  \th(p)=0\}.
\]

It was proved in~\cite{Kesten80} that the square lattice has critical
probability
$\frac12$, and the principle ingredient of the proof is the property
of self-duality.
The dual of the triangular lattice is the hexagonal lattice, and a
further ingredient is
required in order to compute the two corresponding critical
probabilities, namely,
the so-called star--triangle transformation. This calculation was
performed in~\cite{Wierman81}.
General accounts of percolation may be found in \cite
{GrimmettPercolation,GrimmettGraphs},
and of aspects of two-dimensional percolation in~\cite{BolRio}.

We turn now to inhomogeneous percolation on the three lattices of
Figure~\ref{figlatticeexamples}.
The edges of the square lattice are naturally divided into two classes
(horizontal and vertical) of parallel edges,
while those of the triangular and hexagonal lattices may be split into
three such classes.
In \textit{inhomogeneous} percolation, one allows the product measure to
have different intensities on different edges, while requiring that any
two parallel edges have the same intensity.
Thus, inhomogeneous percolation on the square lattice has two parameters,
$\ph$ for horizontal edges and $\pv$ for vertical edges.
We denote the corresponding measure $\PP_{\bp}^{\square}$ where $\bp
=(\ph,\pv)$.\vadjust{\goodbreak}
On the triangular and hexagonal lattices, the measure is defined by a
triplet of parameters $\bp=(p_0, p_1, p_2)$,
and we denote these measures $\PP_{\bp}^\tri$ and $\PP_{\bp}^\hex
$, respectively.

These models have percolation probabilities and critical surfaces, and
the latter were
given explicitly in~\cite{GrimmettPercolation,Kestenbook}.
Let
\begin{eqnarray}
\kappa_{\square}(\bp) &=& \ph+\pv-1,\qquad \bp=(\ph,\pv), \\
\kappa_\tri(\bp) &=& p_0+p_1+p_2 - p_0 p_1 p_2 - 1, \qquad\bp
=(p_0,p_1,p_2),\label{criticaltriangular}\\
\kappa_\hex(\bp) &=& -\kappa_\tri(1-p_0,1-p_1,1-p_2),\qquad \bp
=(p_0,p_1,p_2).
\end{eqnarray}
The following theorem was predicted
in~\cite{SykesEssam}, and discussed in~\cite{Kestenbook}, Section~3.4,
where part (a) was proved and
examples presented in support of parts (b) and (c).
The complete proof of the theorem may be found in~\cite{GrimmettPercolation}, Section~11.9.

\begin{thmm} \label{inhomogeneouscriticality}
The critical surfaces of inhomogeneous bond percolation are given as follows:
\begin{longlist}[(a)]
\item[(a)]\textit{Square lattice}: $\kappa_{\square}(\bp)=0$.
\item[(b)]\textit{Triangular lattice}: $\kappa_\tri(\bp)=0$.
\item[(c)]\textit{Hexagonal lattice}: $\kappa_\hex(\bp)=0$.
\end{longlist}
\end{thmm}

For each of these three processes, the radii and volumes of open clusters
have exponential tails when $\kappa(\bp)<0$, as in~\cite{GrimmettPercolation}, Theorems 5.4, 6.75.

We call a triplet $\bp=(p_0, p_1, p_2)\in[0,1)^3$ \textit{self-dual}
if it satisfies
$\kappa_\tri(\bp)=0$.
By Theorem~\ref{inhomogeneouscriticality},
self-dual points are also critical points, but the neutral
term ``self-duality'' is chosen in order to emphasize that the methods of
this paper do not make use
of criticality \textit{per se\/}.

We write $\a\pm\bp$ for the triplet $(\a\pm p_0,\a\pm p_1,\a\pm p_2)$,
and also $\NN=\{1,2,\dots\}$ for the natural numbers, $\ZZ_0=\NN
\cup\{0\}$, and $\ZZ=\{\dots,-1,0,1,\dots\}$
for the integers.

\subsection{Main results}\label{sectresults}

We concentrate here on the probabilities of open crossings of boxes.
Consider bond percolation
on a connected planar graph $G$ embedded in the plane $\RR^2$.
(For definiteness, we assume that the embedding is
planar, and the edges of $G$ are embedded
as straight line-segments.)
Let $h,l>0$, and let $S_{h,l}= [0,h]\times[0,l]$,
viewed as a subset of $\RR^2$. A \textit{box $S$ of size} $h\times l$ is
a subset of $\RR^2$ of the form $f(S_{h,l})$ for some map $f\dvtx \RR^2
\to\RR^2$
comprising a rotation and a translation.
The box $S=f(S_{h,l})$ is said to \textit{possess open crossings} if
there exist two open paths of $G$ (viewed as
arcs in the plane) whose intersection with $S$ are arcs having one endpoint
in each of the sets $f(\{0\}\times[0,l])$ and $f(\{h\}\times[0,l])$
[resp.,
each of the sets $f([0,h]\times\{0\})$ and $f([0,h]\times\{l\})$].

\begin{definition}\label{defbxp}
Let $G=(V,E)$ be a countable connected graph embedded in the plane.
We say that a measure\vadjust{\goodbreak} $\PP$ on $\Omega=\{0,1\}^E$ has the \textit
{box-crossing property}
if, for $\alpha>0$, there exists $\delta>0$ such that
for all large $N \in\NN$ and any box $S$ of size $\alpha N \times N$,
$S$ possesses
open crossings with probability at least $\delta$.
\end{definition}

Note that the box-crossing property depends on the embedding. The
lattices considered
here will
be embedded in $\RR^2$ in very simple ways that will of themselves
cause no difficulty
in the current context.
For the sake of definiteness at this point, the square lattice
is embedded in $\RR^2$ with vertex-set~$\ZZ^2$, the triangular lattice
has vertex-set $\{a\bi+ b \bj\dvtx  a,b \in\ZZ\}$ where
$\bi=(0,2)$ and $\bj=(1,\sqrt3)$, and the hexagonal lattice is the
dual of the triangular lattice
(with vertices at the circumcenters
of triangles). Note that the box-crossing property is invariant under
affine maps of~$\RR^2$.
It is standard that homogeneous bond percolation on the square lattice
with parameter $p=\frac12$ has the box-crossing property.
(See, e.g.,~\cite{GrimmettPercolation}, Section~11.7 and
Proposition~\ref{RSW}.)

Here is a simplified version of our main result;
more general versions may be found at
Theorems~\ref{generaltriangularRSW} and~\ref{generalsquareRSW}.

\begin{thmm} \label{mainresult}
\begin{longlist}[(a)]
\item[(a)] If $\bp\in(0,1)^2$ satisfies $\kappa_{\square}(\bp)=0$,
then $\PP_{\bp}^{\square}$ has the box-crossing property.
\item[(b)] If $\bp\in[0,1)^3$ satisfies $\kappa_\tri(\bp)=0$,
then both $\PP_{\bp}^\tri$ and $\PP_{1-\bp}^\hex$
have the box-crossing property.
\end{longlist}
\end{thmm}

Square-lattice percolation may be obtained from triangular-lattice
percolation by setting
one of its three parameters to $0$. Thus, part (b) includes part (a).
We choose
to distinguish the two cases, since this is in harmony with the method
of proof.

It is shown in Section~\ref{sectcriticalityusingRSW}
(see Propositions~\ref{sub-criticalityviaRSW}--\ref{propsuper})
that, if a percolation process and its dual process
both possess the box-crossing property, then they are critical. This
provides an
alternative proof
of Theorem~\ref{inhomogeneouscriticality} via Theorem~\ref{mainresult}.
There are several possible choices for the details of the proof, but
the core
arguments comprise the combination of box-crossings, together
with positive association and some type of sharp-threshold statement.
These arguments are extended to a wider class of models next.

\subsection{Highly inhomogeneous models}\label{sechim}

The inhomogeneous models possess translation-invariance but not
rotation-invariance, a
gap that may be spanned by the box-crossing property. Full
translation-invariance is in
fact inessential, and our
arguments may be applied to certain ``highly inhomogeneous models'' that
we describe next.
The following is included as
a demonstration of the use of the box-crossing property, and of the
connection between
the current
work and the geometry of isoradial graphs
(see also~\cite{Ken02}).
This connection will be developed
in a future paper by the current authors.\vadjust{\goodbreak}

\begin{figure}

\includegraphics{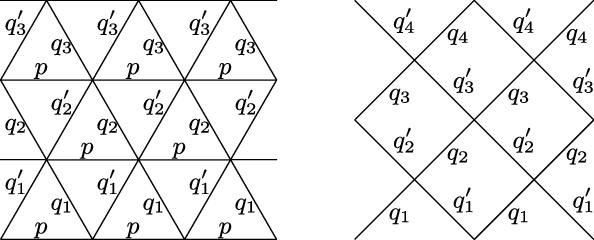}

\caption{\textit{Left:} The triangular lattice with the highly
inhomogeneous product measure $\PP^\triangle_{p, \q, \q'}$.
The probability for each edge to be open is described in the picture:
all horizontal edges have probability~$p$ of being open,
while the other edges have probability $q_n$ (right edges of upward
pointing triangles)
or $q'_n$ (left edges of upward pointing triangles) of being open,
with $n$ being their height.
\textit{Right:} The square lattice with a highly inhomogeneous product
measure $\PP^\square_{\q, \q'}$,
rotated by $\pi/4$.
Edges inclined at angle $\pi/4$ have probability $q_n$ of being open,
while edges inclined at angle $3 \pi/4$ have probability $q'_n$ of
being open,
with $n$ being their height.}
\label{fighighlyinhomogeneousmeasures}
\end{figure}

Let $p \in(0, 1)$, and let $\q= (q_n\dvtx n \in\ZZ) \in[0,1]^\ZZ$ and
$\q' = (q_n' \dvtx  n \in\ZZ) \in[0,1]^\ZZ$.
These are the parameters of our highly inhomogeneous models, and one
such parameter is
allocated to any given edge
of the square and triangular lattices. This is illustrated in Figure
\ref{fighighlyinhomogeneousmeasures}.
This family of models contains
a number of rather natural processes, two examples of which are
presented in Figure~\ref{fighighly}.

\begin{figure}[b]

\includegraphics{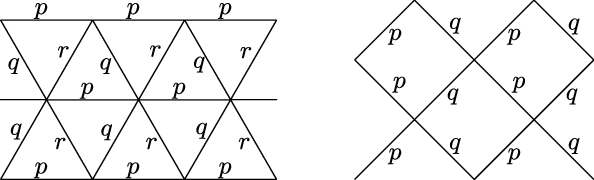}

\caption{Each graph, when repeated periodically, constitutes a model
included within the analysis of highly inhomogeneous models.
\textit{Left:} A triangular-lattice model with a
horizontal axis of symmetry, as in~\cite{Kestenbook}, Section~2.1.
\textit{Right:} A variant on the square-lattice model.}
\label{fighighly}
\end{figure}

Consider first the triangular lattice, and write $\PP^\tri_{p,\q, \q'}$
for the product measure on~$\Om$ under which
any horizontal edge is open with probability $p$ and any right (resp.,
left) edge
of an upward pointing triangle
is open with probability $q_n$ (resp., $q_n'$).
Here, $n\in\ZZ$ denotes the height of the edge as drawn in the figure.
Let $\PP^\hex_{1-p, 1-\q, 1-\q'}$ be the measure on the hexagonal
lattice that is dual to
$\PP^\tri_{p, \q, \q'}$.

Consider next the square lattice. The measure $\PP^{\square}_{\q, \q'}$
is defined similarly
to the above, as in Figure~\ref{fighighlyinhomogeneousmeasures}.
We refer to the three probability measures thus defined as \textit{highly inhomogeneous}.
Under suitable conditions on their parameters,
each may be regarded as percolation on an isoradial graph with edge-parameters
chosen in the canonical way according to the isoradial embedding (see
\cite{Ken02}, Section~5).

We may show, under suitable conditions, that highly inhomogeneous
models have the box-crossing property.

\begin{thmm}\label{generaltriangularRSW}
If $p\in(0,1)$, $\q,\q'\in[0,1]^\ZZ$ satisfy
%
\begin{equation}
\forall n \in\ZZ \qquad\kappa_\tri(p,q_n,q_n')=0, \label
{criticaltriangular2}
\end{equation}
then both $\PP^\triangle_{p, \q, \q'}$ and $\PP^{\hexagon}_{1-p,
1-\q, 1-\q'}$ have the box-crossing property.
\end{thmm}

\begin{thmm}\label{generalsquareRSW}
Let $\q,\q'\in(0,1)^\ZZ$. If there exists $\eps> 0$ such that
%
\begin{equation}
\forall n \in\ZZ\qquad \kappa_{\square}(q_n,q_n')=0\quad \mbox{and}\quad
q_n, q_n' \ge\varepsilon, \label{criticalsquare2}
\end{equation}
then $\PP^\square_{\q, \q'}$ has the box-crossing property.
\end{thmm}

The reader is reminded in Propositions~\ref{sub-criticalityviaRSW}
and~\ref{propsuper}
that the box-crossing property implies power-law behavior of the
radius of an open cluster.

An assumption along the lines of the
second condition on $\q$, $\q'$ in \eqref{criticalsquare2} is
necessary: if, for example, $q_n$
is sufficiently small over an interval of values of $n$,
then the chance of crossing a certain diagonally oriented rectangle is
correspondingly small, and thus the box-crossing property could not hold.
From the above theorems may be obtained a characterization of the
critical surface of a highly inhomogeneous model.
We call a percolation measure $\PP$ \textit{uniformly supercritical} if
there exists $\theta>0$
such that $\PP(v \lra\oo) \ge\theta$ for every vertex $v$. The
open cluster at a vertex $v$,
written $C_v$, is the set of vertices
joined to $v$ by open paths.

\begin{thmm}\label{generaltriangularcriticality}
Let $p \in(0,1)$ and $\q,\q' \in[0,1)^\ZZ$.
\begin{longlist}
\item[(a)] If
\begin{equation}
\forall n \in\ZZ\qquad \kappa_\tri(p,q_n,q_n') \leq0,
\end{equation}
then there exists, $\PP^\triangle_{p, \q, \q'}$-a.s., no infinite
open cluster.
\item[(b)]
If there exists $\delta> 0$ such that
\begin{equation}\label{G17}
\forall n \in\ZZ\qquad \kappa_\tri(p,q_n,q_n') \le-\delta,
\end{equation}
then there exist $c,d>0$ such that, for every vertex $v$,
\[
\PP^\tri_{p,\q,\q'}(|C_v| \ge k) \le ce^{-dk},\qquad  k \ge0.
\]
\item[(c)] If there exists $\delta> 0$ such that
\begin{equation}\label{generaltriangularsuper-criticalitycondition}
\forall n \in\ZZ\qquad \kappa_\tri(p,q_n,q_n') \geq\delta,
\end{equation}
then $\PP^\triangle_{p, \q, \q'}$ is uniformly supercritical.
\end{longlist}
The same holds for $\PP^{\hexagon}_{p,\q,\q'}$ with $\kappa_\hex$
in place of $\kappa_\tri$.
\end{thmm}

\begin{thmm}\label{generalsquarecriticality}
Let $\q,\q' \in(0,1)^\ZZ$.
\begin{longlist}
\item[(a)]
If there exists $\eps> 0$ such that
%
\begin{equation}
\forall n \in\ZZ\qquad\kappa_{\square}(q_n,q_n')\le0\quad \mbox{and}\quad
q_n, q_n' \le1-\varepsilon, \label{G15}
\end{equation}
then there exists, $\PP^\square_{\q, \q'}$-a.s., no infinite open cluster.
\item[(b)]
If there exists $\de> 0$ such that
\[
\forall n \in\ZZ\qquad \kappa_{\square}(q_n,q_n')\le-\de, %
\]
then there exist $c,d>0$ such that, for every vertex $v$,
\[
\PP^{\square}_{\q,\q'}(|C_v| \ge k) \le ce^{-dk},\qquad  k \ge0.
\]
\item[(c)] If there exists $\delta> 0$ such that, for all $n$,
\[
\kappa_{\square}(q_n,q_n' ) \ge\delta,
\]
then $\PP^\square_{\q, \q'}$ is uniformly supercritical.
\end{longlist}
\end{thmm}

Parts (c) of the above two theorems give conditions under which there
exists an infinite open
cluster. Such a cluster is necessarily (almost surely) unique since the
box-crossing property implies the
existence of open cycles in annuli.

We will prove Theorem~\ref{mainresult} in Section~\ref{sectmainpf},
and Theorems~\ref{generaltriangularRSW}--\ref{generalsquarecriticality}
in Section~\ref{sectgeneralizations}.

\section{Star--triangle transformation}\label{secttransformation}
\subsection{The basic transformation}\label{sectsimpletransformation}

The star--triangle transformation was discovered first in the context
of electrical networks,
and adapted by Onsager and Kramers--Wannier to the Ising model. In its
base form, it is
a graph-theoretic transformation between the hexagonal lattice and the
triangular lattice.
Its importance stems from the fact that a variety of probabilistic
models are
conserved under this transformation, including the critical
percolation, Potts and random-cluster
models. The methods of this paper extend to all such systems, but we
concentrate here on
percolation, for which we summarize its manner of operation as in~\cite{GrimmettPercolation}, Section~11.9.

\begin{figure}

\includegraphics{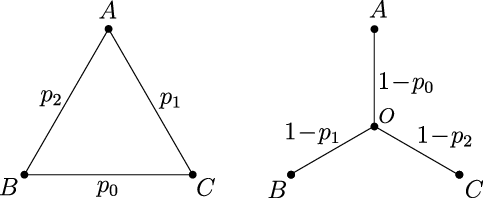}

\caption{The star--triangle transformation.}
\label{figstartriangletransformation}
\end{figure}

Consider the triangle $G=(V,E)$
and the star $G'=(V',E')$ drawn in Figure \ref
{figstartriangletransformation}. Let $\bp=(p_0,p_1,p_2)$.
In order to eliminate trivialities,\vadjust{\goodbreak} we shall assume throughout this
paper that $\bp\in[0,1)^3$.
Write $\Om=\{0,1\}^E$ with associated product probability measure $\PP
_\bp^\tri$,
and $\Om'=\{0,1\}^{E'}$ with associated measure $\PP_{1-\bp}^\hex$.
Let $\omega\in\Om$ and
$\omega'\in\Om'$.
For each graph we may consider open connections between its vertices,
and we abuse notation by writing, for example, $x \displaystyle\mathop{\xleftrightarrowt}^
{G,\omega} y$ for the \textit{indicator
function} of the event that $x$ and $y$ are connected by an open path
of $\omega$.
Thus, connections in $G$ are described by the family
$( x \displaystyle\mathop{\xleftrightarrowt}^{G , \omega} y\dvtx  x,y \in V)$ of random variables,
and similarly for $G'$.

\begin{prop}[(Star--triangle transformation)] \label{simplestartriangle}
Let $\bp\in[0,1)^3$ be self-dual in the sense that $\kappa_\triangle
(\bp)=0$.
The families
\[
( x \mathop{\xleftrightarrowt}^{G, \omega} y \dvtx  x,y = A,B,C),\qquad
( x \mathop{\xleftrightarrowt}^{G', \omega'} y \dvtx x,y = A,B,C),
\]
have the same law.
\end{prop}

The proof may be\vspace*{1pt} found in~\cite{GrimmettPercolation}, Section~11.9.
It will be helpful in the following to explore natural couplings of the
two measures of the
proposition. Let $\bp\in[0,1)^3$ be self-dual, and let $\Om$ (resp., $\Om'$)
have associated measure $\PP_\bp^\tri$ (resp., $\PP_{1-\bp
}^\hex$)
as above. There exist random mappings $T\dvtx \Om\to\Om'$ and
$S\dvtx  \Om'\to\Om$ such that $T(\omega)$ has law $\PP_{1-\bp}^\hex
$, and
$S(\omega')$ has law $\PP_\bp^\tri$. Such mappings are given in
Figure~\ref{figsimpletransformationcoupling},
and we shall not specify them more formally here. Note from
the figure that $T(\omega)$ is
deterministic for seven of the eight elements of~$\Om$; only in the
eighth case does $T(\omega)$ involve
further randomness. Similarly, $S(\omega')$ is deterministic except
for one special $\omega'$.
Each probability in the figure is well defined since
$P := (1-p_0)(1-p_1)(1-p_2)>0$.

\begin{figure}

\includegraphics{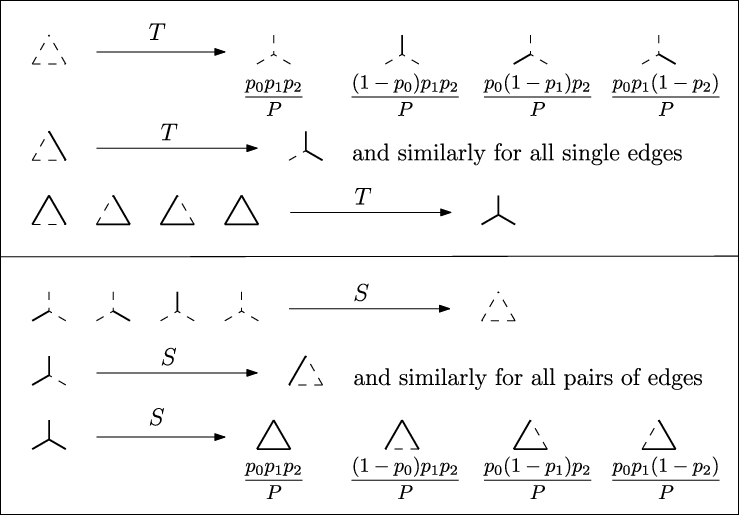}

\caption{The ``kernels'' $T$ and $S$ and their
transition probabilities, with $P:=(1-p_0) (1-p_1)\times (1-p_2)$.
Since $\kappa_\tri(\bp)=0$,
the probabilities in the first and last rows sum to $1$.}
\label{figsimpletransformationcoupling}
\end{figure}

\begin{prop}[(Star--triangle coupling)]\label{propst-coupling}
Let $\bp$ be self-dual and let $S$ and~$T$ be given as in Figure \ref
{figsimpletransformationcoupling}.
With $\omega$ and $\omega'$ sampled as above:
\begin{longlist}
\item[(a)]
$T(\omega)$ has the same law as $\omega'$,
\item[(b)]$S(\omega')$ has the same law as $\omega$,
\item[(c)] for $x,y \in\{ A,B,C \}$, $x \displaystyle\mathop{\xleftrightarrowt}^{G,\omega} y$
if and only if $x \displaystyle\mathop{\xleftrightarrowi}^{G',T(\omega)} y$,
\item[(d)] for $x,y \in\{ A,B,C \}$, $x \displaystyle\mathop{\xleftrightarrowt}^{G',\omega'} y$
if and only if $x \displaystyle\mathop{\xleftrightarrowi}^{G,S(\omega')} y$.\vadjust{\goodbreak}
\end{longlist}
\end{prop}

\subsection{Transformations of lattices}

We show next how to use the star--triangle transformation to transform
the triangular lattice into the square lattice.
This transformation may be extended
to transport self-dual measures on the first lattice to measures on the
second lattice,
via a coupling that preserves open connections.
This permits the transportation of the box-crossing property from one
lattice to the other.
This general approach was introduced by Baxter and Enting~\cite{Baxter399}
in a study of the Ising model, and has since been developed under the
name Yang--Baxter equation,
\cite{McCoybook,Perk-AY}.

Henceforth, it is convenient to work with so-called \textit{mixed
lattices} that combine the square
lattice with either the triangular or hexagonal lattice. We shall be
precise about the manner
in which a mixed lattice is embedded in $\RR^2$.
Let $i\in\RR$, and let $I=\RR\times\{i\}$ be the horizontal line of
$\RR^2$
with \textit{height} $i$, called the \textit{interface};
above $I$ consider the triangular lattice and below $I$ the square lattice.
Our triangular lattice comprises equilateral triangles with side length
$\sqrt{3}$, and
our square lattice comprises rectangles whose horizontal (resp.,
vertical) edges have length~$\sqrt{3}$
(resp.,~$1$), as illustrated in the leftmost diagram of
Figure~\ref{figlatticetransformation}. The embedding is
specified
up to horizontal translation and, in order to be precise, we assume
that the point $(0,i)$ is a vertex of the lattice.
We call the ensuing graph the \textit{mixed triangular lattice}
$\Lattice$ with interface $I=I_\Lattice$.

\begin{figure}

\includegraphics{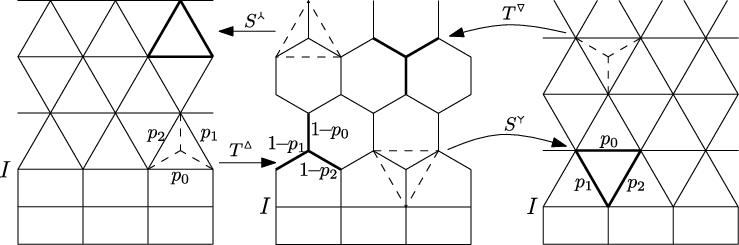}

\caption{Transformations $\Su$, $\Sd$, $\Tu$ and $\Td$ of mixed lattices.
The transformations map the zones with dashes to
the bold triangles/stars.
The interface-height decreases by $1$ from the leftmost to the
rightmost graph.}
\label{figlatticetransformation}
\end{figure}

The \textit{mixed hexagonal lattice} $\Lattice$ with interface
$I=I_\Lattice$ is similarly composed of
a regular hexagonal lattice (of side length $1$) above $I$
and a square lattice below $I$ (with edge-lengths as above), as drawn
in the central diagram
of Figure~\ref{figlatticetransformation}.

We define the \textit{height} $h(A)$ of a subset $A \subseteq\RR^2$
as the supremum of the $y$-coordinates of elements of $A$.
A mixed lattice $\Lattice$ may be identified with the subset of $\RR
^2$ belonging to its edge-set.
Thus, for a mixed lattice $\Lattice$, $h(I_\Lattice)$ is the height
of its interface.

We next define two transformations$, \Tu$ and $\Td$ acting on a mixed
triangular lattice $\Lattice$:
\begin{longlist}[(a)]
\item[(a)]$\Tu$ transforms all upward pointing triangles of $\Lattice$
into stars,
with centers at the circumcenters of the equilateral triangles.
\item[(b)]$\Td$ transforms all downward pointing triangles into stars.
\end{longlist}
It is easily checked (and illustrated in Figure \ref
{figlatticetransformation})
that each transformation maps a mixed triangular lattice
to a mixed hexagonal lattice.

We define similarly the transformations $\Su$ and $\Sd$ on a mixed
hexagonal lattice;
these transform all upward (resp., downward) pointing stars into
triangles.
They transform a mixed hexagonal lattice to a mixed triangular lattice.

The concatenated operators $\Su\comp\Td$ and $\Sd\comp\Tu$ map
the mixed triangular lattice $\Lattice$
to another mixed triangular lattice, but with a different interface height:
\begin{eqnarray*}
h( I_{\Su\comp\Td\Lattice} ) &=& h(I_\Lattice) + 1, \\
h( I_{\Sd\comp\Tu\Lattice} )& =& h(I_\Lattice) - 1.
\end{eqnarray*}
Loosely speaking, repeated application of $\Su\comp\Td$
transforms $\Lattice$ into the square lattice,
while repeated application of $\Sd\comp\Tu$ transforms it into the
triangular lattice.

We now extend the domains of the above maps to include configurations.
Let $\Lattice=(V,E)$ be a mixed triangular lattice with $\Om_E=\{0,1\}^E$,
and let $\omega\in\Om_E$. The image of $\Lattice$ under $\Tu$ is written
$\Tu\Lattice= (\Tu V, \Tu E)$ and we write $\Om_{\Tu E} = \{0,1\}
^{\Tu E}$.
Let $\bp\in[0,1)^3$.
Let $\Tu(\omega)$ be chosen (randomly) from $\Om_{\Tu E}$
by independent application of the kernel $T$ within every upward
pointing triangle of $\Lattice$.
Note that the random map $T$ depends on the choice of $\bp$.

By Proposition~\ref{propst-coupling}, for any two vertices $A$,
$B$ on $\Lattice$, we have
\begin{equation}
(A \mathop{\xleftrightarrowt}^{\Lattice,\omega} B )
\quad\Leftrightarrow\quad
\bigl(A \mathop{\xleftrightarrowi}^{\Tu\Lattice, \Tu(\omega)}B\bigr).
\label{sameconnections}
\end{equation}
The corresponding statements for $\Td$, $\Su$ and $\Sd$ are valid also,
with one point of note.
In applying the transformations $\Su$, $\Sd$ to a mixed hexagonal lattice,
the points $A$ and $B$ in the corresponding versions of \eqref
{sameconnections} must not be centers of
transformed stars,
since these points disappear during the transformations.

Let $\bp=(p_0,p_1,p_2)\in[0,1)^3$ be self-dual, and let $\Su$, $\Sd
$, $\Tu$, $\Td$ be given accordingly.
We identify next the probability measures on the mixed lattices that
are preserved
by the operation of these transformations.

Let $\Lattice=(V,E)$ be a mixed (triangular or hexagonal) lattice. The
probability measure denoted
$\PP_\bp$ on $\Om_E$ is a product measure whose intensity $p(e)$ at
edge $e$ is given as follows:
\begin{longlist}[(a)]
\item[(a)]$p(e) = p_0$ if $e$ is horizontal,
\item[(b)]$p(e) = 1- p_0$ if $e$ is vertical,
\item[(c)]$p(e) = p_1$ if $e$ is the right edge of an upward pointing triangle,
\item[(d)]$p(e)=p_2$ if $e$ is the left edge of an upward pointing triangle,
\item[(e)]$p(e) = 1-p_2$ if $e$ is the right edge of an upward pointing star,
\item[(f)]$p(e)=1-p_1$ if $e$ is the left edge of an upward pointing star.
\end{longlist}
When it becomes necessary to emphasize the lattice $\Lattice$ in
question, we shall write~$\PP_\bp^{\Lattice}$.

\begin{prop}\label{P-invariant}
If $\bp\in[0,1)^3$ is self-dual in that $\kappa_\tri(\bp)=0$, then
$\PP_\bp$ is
preserved by the transformations $\Su$, $\Sd$, $\Tu$k and $\Td$.
That is,
if $U$ is any of these four transformations acting on the mixed lattice
$\Lattice=(V,E)$, then
\[
\mbox{$\omega\in\Om_E$ has law $\PP_\bp^\Lattice$}\quad
\Leftrightarrow\quad\mbox{$U(\omega)$ has law
$\PP_\bp^{U\Lattice}$}.
\]
\end{prop}

\subsection{Transformations of paths}

Since the star--triangle transformation preserves open connections
[cf. \eqref{sameconnections}], there is a sense in which it maps
open paths to open paths.\vadjust{\goodbreak}
Thus, if percolation on a mixed lattice $\Lattice$ has the
box-crossing property, one expects
that its image also has the box-crossing property. Some difficulties occur
in the proof of this, arising from the fact that the image of an open path
tends to drift away from the original. We study this drift next. It is
convenient to work with general paths in $\RR^2$.

Recall that a \textit{path} $\Ga=(\Ga_t)$ in $\RR^2$ is a
continuous function
$\Ga\dvtx [a,b] \to\RR^2$ for some real interval $[a,b]$. Note that a path
$\Ga$ may in general have self-intersections, and there may be
subintervals of
$[a,b]$ on which $\Ga$ is constant. Let
$\phi\dvtx  [c,d] \to[a,b]$ be continuous and strictly increasing with
$\phi(c)=a$ and $\phi(d)=b$. We term the path $\Ga_\phi=(\Ga_{\phi
(t)})$ a \textit{reparametrization} of $\Ga$
over $[c,d]$.

Let $|\cdot|$ denote the Euclidean norm on $\RR^2$.
The space of paths may be metrized by
\[
d(\Gamma, \Pi) = \inf\Bigl\{
\sup_{t \in[0,1]} | \Gamma'_t - \Pi'_t |\Bigr\},
\]
where the infimum is over all reparametrizations $\Ga'$ (resp.,
$\Pi
'$) of $\Ga$
(resp.,~$\Pi$) over $[0,1]$.
Note that $d$ is not a metric since $d(\Ga,\Ga')=0$ if
$\Ga'$ is a reparametrization of $\Ga$, and thus the corresponding
metric acts on a space of equivalence classes of paths
(see~\cite{AizBur}, equation (2.1)).
We shall use the fact that, if two paths (parametrized
over $[0,1]$) satisfy
$d(\Gamma,\Pi) < \delta$, then
\[
\Gamma\subseteq\Pi^{\delta},\qquad
| \Gamma_0 - \Pi_0| \leq\delta,\qquad
| \Gamma_1 - \Pi_1 | \leq\delta,
\]
where
\[
A^\de:= \{x+y\dvtx  x \in A, |y| \le\de\}.
\]

Henceforth, all paths will be lattice-paths (we allow loops and
repeated edges).
Such a path is called \textit{open} (in a given configuration) if it
traverses only
open edges.

Let $\omega$ be an edge-configuration on a mixed triangular lattice
$\Lattice$.
Let $\Ga$ be an $\omega$-open lattice-path of $\Lattice$, and
consider the action of the map
$\Tu$ (illustrated in Figure~\ref{figpathtransformation}).
The image lattice $\Tu\Lattice$ is endowed with
the edge-configuration $\Tu(\omega)$, and we
explain next the construction of a $\Tu(\omega)$-open path $\Tu
(\Gamma)$ on $\Tu\Lattice$.
The path $\Tu(\Gamma)$ will remain close to $\Ga$,
and it will depend only locally on $\Gamma$ and $\omega$.

We summarize the argument for $\Tu(\Ga)$ [the same argument is valid
for $\Td(\Ga)$].
The path $\Ga$ passes through the sequence $\g_0,\g_1,\dots,\g_m$
of vertices
of $\Lat$, in order. Since $\Lat$ is a mixed triangular lattice, each
$\g_i$ is
present in $\Tu\Lat$ also. The edge $\g_i\g_{i+1}$ of $\Lat$ lies
either in
its square part (excluding the interface) or its triangular part
(including the interface). If the former, it lies also in $\Tu\Lat$.
If the latter, it lies in a unique upward pointing triangle $t$ of
$\Lat$. Under $\Tu$,
$t$ is mapped to a star $\Tu(t)$, and the configuration on $t$ is
mapped to a configuration on
$\Tu(t)$ in which $\g_i$ is connected by an open path to $\g_{i+1}$
via the center $O$ of the star.
We replace the edge $\g_i\g_{i+1}$ of $\Lat$ by this open path.
This is done for each edge of $\Ga$, and the outcome, denoted $\Tu
(\Ga)$, is an open path
of $\Tu\Lat$ with the same endpoints as $\Ga$.\vadjust{\goodbreak}

We turn now to a mixed hexagonal lattice $\LatticeH$ under the
transformation $\Sd$
(the same argument holds for $\Su$).
Let $\omega$ be an edge-configuration
on $\LatticeH$, and $\Ga$ an open path. Readers may be content with
the illustration
of Figure~\ref{figpathtransformation}, but further details are given below:

\begin{figure}

\includegraphics{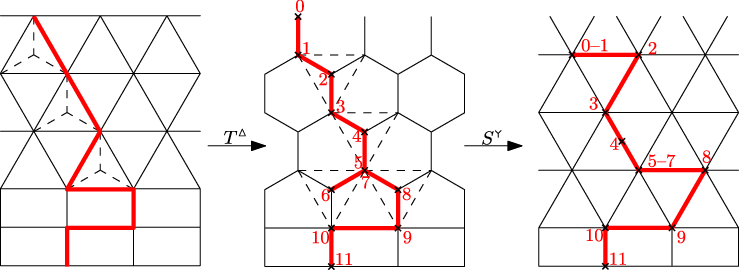}

\caption{Transformations of lattice-paths.
The transformation $\Tu$ acts deterministically on open paths,
each edge of a triangle being transformed into two segments of an
upward pointing star.
When applying $\Sd$,
the segment labeled from $0$ to $1$ contracts to one point,
as does that labeled from $5$ to $7$.}
\label{figpathtransformation}
\end{figure}

We parametrize $\Gamma$ as $(\Ga_t\dvtx  0 \le t \le N)$ in such
a way that:
\begin{longlist}[(a)]
\item[(a)]$(\Gamma_n \dvtx  n =0,1, \dots,N)$ are the vertices visited by
$\Gamma$
in sequence (possibly with repetition),
\item[(b)] each $\Gamma_{[n, n+1]}$ is either an edge or a vertex of the
lattice, and
\item[(c)]$\Gamma$ is affine on the intervals $[n,n+1]$.
\end{longlist}

It suffices to define the image under $\Sd$ of each edge of $\Ga$.
For simplicity,
we assume that $\Ga$ has no stationary points; the argument is exactly
similar otherwise.
Let~$g_n$ be the edge $\Ga_{[n,n+1]}$ of $\LatticeH$.
If $g_n$ lies in the square part of $\LatticeH$ (i.e., in or below
the interface), we set
$\Sd(\Ga)_n =\Ga_n$ and $\Sd(\Ga)_{n+1} = \Ga_{n+1}$, with linear
interpolation
between.

Suppose that $g_n$ lies in the hexagonal part of $\LatticeH$, so that
$g_n$ is an edge
of a downward pointing star whose exterior vertices we denote as $A$,
$B$, $C$, and
whose central vertex as $O$. Thus, $g_n$ has $O$ as one endvertex, and
its other endvertex
lies in $\{A, B, C\}$.
\begin{enumerate}[(a)]
\item[(a)] If $n = 0$ and $\Ga_0=O$,
set $\Sd(\Gamma)_{[0,1]} = \Gamma_1$.
\item[(b)] If $n = N-1$ and $\Ga_N=O$,
set $\Su(\Gamma)_{[N-1,N]} = \Gamma_{N-1}$.
\item[(c)] Suppose $n \ge1$ and $\Ga_n = O$ (a similar argument holds if
$n \le N-2$ and $\Ga_{n+1} =O$).
Then $\Ga_{n-1}, \Ga_{n+1} \in\{A,B,C\}$.
\begin{itemize}
\item If $\Gamma_{n-1} = \Gamma_{n+1}$, set $\Sd(\Gamma)_{[n-1,
n+1]} = \Gamma_{n-1}$.
\item If $\Gamma_{n-1} \neq\Gamma_{n+1}$
and the edge $\Gamma_{n-1} \Gamma_{n+1}$ is open in $\Sd(\omega)$, set
$\Sd(\Gamma)_{n-1} = \Gamma_{n-1}$ and $\Sd(\Ga)_{n+1} =\Gamma
_{n+1}$, with linear interpolation between.\vadjust{\goodbreak}
\item Suppose $\Gamma_{n-1} \neq\Gamma_{n+1}$
and the edge $\Gamma_{n-1}\Gamma_{n+1}$ is closed in $\Sd(\omega)$,
and let $C$ denote the third exterior vertex of the star in question.
The edges $\Ga_n C$ and $\Ga_{n+1}C$ are necessarily open in $\Sd
(\omega)$.
Set $\Sd(\Gamma)_{n-1} =\Gamma_{n-1}$,
$\Sd(\Gamma)_{n} = C$,
$\Sd(\Gamma)_{n+1} = \Gamma_{n+1}$,
with linear interpolation between.
\end{itemize}
\end{enumerate}

\begin{prop}\label{distancebetweenpaths}
Let $\Gamma$ be a path of a mixed lattice. We have that:
\begin{longlist}[(a)]
\item[(a)]$d(\Gamma, \Tu(\Gamma)) \leq\frac12$ and $d(\Gamma, \Td
(\Gamma)) \leq\frac12$,
\item[(b)]$d(\Gamma, \Su(\Gamma)) \leq1$ and $d(\Gamma, \Sd(\Gamma))
\leq1$,
\item[(c)]$d(\Gamma, (\Su\comp\Td) (\Gamma)) \leq1$ and $d(\Gamma,
(\Sd\comp\Tu)(\Gamma)) \leq1$,
\end{longlist}
whenever the transformations are matched to the mixed lattice.
\end{prop}

\begin{pf} 
This follows by examination of the cases above, and is illustrated
in Figure~\ref{figpathtransformation}.
\end{pf}

\section{\texorpdfstring{Proof of Theorem \protect\ref{mainresult}}{Proof of Theorem 1.3}}\label{sectmainpf}

\subsection{The box-crossing property}\label{srecbxp}
Whereas the box-crossing property of Definition~\ref{defbxp} involves
crossing of
boxes with arbitrary orientations,
it is in fact necessary and sufficient that boxes with sides parallel
to the axes
possess horizontal and vertical open
crossings with probabilities bounded away from $0$.

Let $\Lattice=(V,E)$ be a lattice embedded in the plane, and let
$\omega\in\Om_E = \{0,1\}^E$. Let $\Ch(m,n)$
[resp., $\Cv(m,n)$] be the event
that there is an open horizontal (resp., vertical) crossing
of the box $B_{m,n} = [ - m , m] \times[0,n
]$ of $\RR^2$.
Suppose now that~$\Lat$ is invariant under translation by the nonzero
real vectors $(a,0)$
and $(0,b)$ for some least positive $a$ and $b$.
A probability measure $\PP$ on $\Om_E$ is called \textit
{translation-invariant}
if it is invariant under the actions of these translations.
It is said
to be \textit{positively associated} if
\[
\PP(A \cap B) \ge\PP(A) \PP(B)
\]
for all increasing events $A$, $B$ (see~\cite{GrimmettRCM}, Section~2.2).
By the Harris--FKG inequality (see~\cite{GrimmettPercolation}, Section~2.2),
product measures are positively associated.

\begin{prop}\label{RSW}
A translation--invariant, positively associated probability
measure $\PP$ on $\Om_E$
has the box-crossing property if and only if the following
hold:
\begin{longlist}[(a)]
\item[(a)] for $\alpha>0$, there exists $\delta>0$ such that, for all
large $N \in\NN$,
\begin{equation}
\PP[ \Ch(\alpha N, N)] > \delta, \label{horizontalcrossing}
\end{equation}
\item[(b)] there exist $\beta, \delta>0$ such that, for all large $N \in
\NN$,
\begin{equation}
\PP[ \Cv(N, \beta N)] > \delta\label{verticalcrossing}.
\end{equation}
\end{longlist}
\end{prop}

It is standard that $\PP^{\square}_{1/2,1/2}$ satisfies the
above conditions and
therefore
has the box-crossing property. See~\cite{Russo,Seymour-Welsh}, and the
accounts in
\cite{GrimmettPercolation}, Section~11.7
and~\cite{GrimmettGraphs}, Section~5.5.

\begin{figure}

\includegraphics{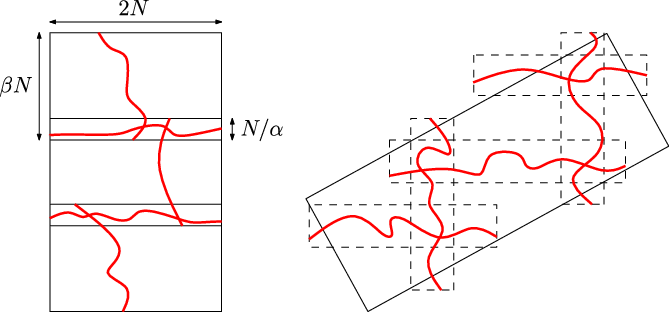}

\caption{\textit{Left:} Vertical crossings of copies of $B_{N,\beta N}$
and horizontal crossings of copies of $B_{N, N/\a}$, for large $\a$,
may be combined to obtain vertical crossings of boxes with arbitrary
aspect ratio.
\textit{Right:} Crossings of the type $\Ch(\g n,n)$ and $\Cv(n,\g n)$
may be combined to obtain crossings of boxes with general
inclination.}
\label{figcombiningcrossings}
\end{figure}

\begin{obs}\label{rem0}
If the measure $\PP$ of Proposition~\ref{RSW} is not
translation--invariant, the
proposition remains valid with \eqref{horizontalcrossing}--\eqref
{verticalcrossing}
replaced by the same inequalities \textit{uniformly} for all translates
of the relevant rectangles.
\end{obs}

\begin{pf} 
This is sketched.
It is trivial that the box-crossing property implies~\eqref
{horizontalcrossing} and
\eqref{verticalcrossing}.
Conversely, suppose \eqref{horizontalcrossing} and \eqref{verticalcrossing}
hold.
The positive association
permits the combination of box-crossings
to obtain crossings of larger boxes.
The claim is now obtained as illustrated in Figure \ref
{figcombiningcrossings}.
\end{pf}

\subsection{\texorpdfstring{Proof of Theorem \protect\ref{mainresult}}{Proof of Theorem 1.3}} \label{secmainprop}

Theorem~\ref{mainresult} is an immediate consequence of the following
theorem. Recall that a triplet $\bp\in[0,1)^3$
is \textit{self-dual} if $\kappa_\tri(\bp)=0$, with~$\kappa_\tri$
given in
\eqref{criticaltriangular}.

\begin{thmm} \label{bxp}
Let $\bp= (p_0,p_1,p_2) \in[0,1)^3$ be self-dual.
\begin{longlist}[(a)]
\item[(a)]
If $\PP_{(p_0, 1-p_0)}^\square$ has the box-crossing property,
then so does $\PP_{\bp}^\triangle$.
\item[(b)]
Let $p_0>0$. If $\PP_{\bp}^\triangle$ has the box-crossing property,
then so does
$\PP_{(p_0, 1-p_0)}^\square$.
\item[(c)]
$\PP_{\bp}^\triangle$ has the box-crossing property if and only if
$\PP_{1-\bp
}^{\hexagon}$ has it.
\end{longlist}
\end{thmm}

Since $\PP_{(1/2, 1/2)}^\square$ has the box-crossing property,
we have by Theorem~\ref{bxp}(a) that
$\PP_{(1/2, p_1, p_2)}^\triangle$
has the box-crossing property for all self-dual triplets $(\frac12,
p_1, p_2)$.
As $(\frac12, p_1, p_2)$ ranges within the set of self-dual triplets,
$p_1$ ranges over
the interval $[0,\frac12]$.
By Theorem~\ref{bxp}(b),
for all $p_1 \in(0,\frac12)$, $\PP_{(p_1, 1-p_1)}^\square$ has the
box-crossing property.
We then use Theorem~\ref{bxp}(a) again to deduce that
$\PP_{\bp}^\triangle$ has the box-crossing property for all self-dual triplets $\bp$.
Finally, the conclusion may be extended to
the hexagonal lattice by Theorem~\ref{bxp}(c).

Theorem~\ref{bxp}(a, b) is proved in the remainder of this section.
Part (c) is an immediate
consequence of a single application of the star--triangle transformation,
and no more will be said about this. We assume henceforth that all
lattices are
embedded in $\RR^2$ in the style of Figure~\ref{figlatticetransformation}.

\subsection{\texorpdfstring{Proof of Theorem \protect\ref{bxp}\textup{(a)}}
{Proof of Theorem 3.3(a)}}\label
{sectproofsquaretotriangle}

It suffices to assume $p_0>0$, since the hypothesis does not hold when $p_0=0$.
By Proposition~\ref{RSW},
it suffices to prove the following two propositions.

\begin{prop} \label{horizontalRSWsquaretotriangle}
Let $\bp=(p_0,p_1,p_2)\in[0,1)^3$ be self-dual with $p_0>0$.
For $\alpha> 1$ and $N \in\NN$,
\[
\PP_\bp^\triangle\bigl[\Ch\bigl((\alpha-1) N,2N \bigr)\bigr] \geq\PP_{(p_0,
1-p_0)}^\square[\Ch(\alpha N, N )].
\]
\end{prop}

\begin{prop} \label{verticalRSWsquaretotriangle}
Let $\bp=(p_0,p_1,p_2) \in[0,1)^3$ be self-dual with $p_0>0$.
There exist $\beta=\b(p_0) > 0$, and $\rho_N = \rho_N(\b)>0$
satisfying $\rho_N\to1$ as $N \to\infty$,
such that
\[
\PP_\bp^\triangle[\Cv( 2N ,\beta N)] \geq
\rho_N \PP_{(p_0,1-p_0)}^\square[\Cv(N, N)],\qquad N\in\NN.
\]
\end{prop}

The constant $\beta$ is given by
%
\begin{equation}\label{G12}
\b:=\frac{1-\sqrt{1-p_0(1-p_0)}}{1-p_0},
\end{equation}
and $\rho_N=\rho_N(\b)$
may be calculated explicitly by the final argument of this subsection.

\begin{pf*}{Proof of Proposition \protect\ref{horizontalRSWsquaretotriangle}}
Let $\bp\in[0,1)^3$ be self-dual with $p_0>0$, and let $\alpha> 1$
and $N \in\NN$.
Let $\Lat= (V,E)$ be a mixed triangular lattice with interface-height
$h(I_\Lattice) = N$,
and write $\PP_\bp$ for the associated product measure on $\Lattice$.
Since $ B_{\alpha N, N} = [ - \alpha N , \alpha N ] \times
[ 0, N ]$
is beneath the interface,
\[
\PP_{(p_0, 1-p_0)}^\square[\Ch(\alpha N, N ) ] = \PP
_\bp^{\Lattice}[\Ch(\alpha N, N )].
\]

Let $\omega\in\Ch(\alpha N, N )$. We claim that
there exists a horizontal open crossing of
$B_{(\alpha-1) N, 2N}$ in $(\Sd\comp\Tu)^N (\omega)$,
as illustrated in Figure~\ref{fighorizontalRSWsquaretotriangle}.

\begin{figure}

\includegraphics{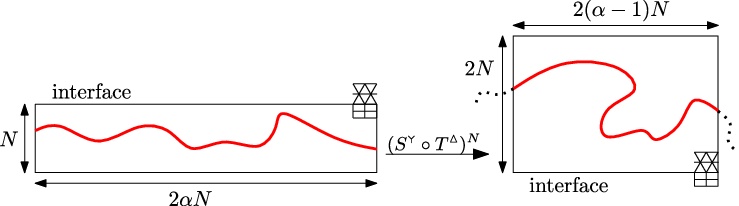}

\caption{Transformation of a horizontal crossing of $B_{\alpha N, N}$
by $(\Sd\comp\Tu)^N$.
The interface moves down $N$ steps.
The path drifts by at most distance $N$ and cannot go below the interface
of the image lattice.}
\label{fighorizontalRSWsquaretotriangle}
\end{figure}

Let $\Gamma$ be an open path of $\Lat$, parametrized by $[0,1]$, that crosses
$B_{\alpha N, N}$ horizontally.
By Proposition~\ref{distancebetweenpaths},\vadjust{\goodbreak}
$d(\Gamma, \Ga(N)) \leq N$ where $\Ga(N):=(\Sd\comp\Tu)^N (\Gamma
)$, whence,
\begin{eqnarray}
&| \Gamma_0 - \Ga(N)_0 | \leq N ,& \label{leftendpoint}\\
&| \Gamma_1 - \Ga(N)_1 | \leq N ,& \label
{rightendpoint}\\
&\Ga(N) \subseteq\Gamma^{N} \subseteq B_{\alpha N, N} ^{N}.&
\end{eqnarray}
Since $\Ga$ contains no vertex with strictly negative $y$-coordinate,
neither does $\Ga(N)$.
Hence,
\[
\Ga(N) \subseteq\Gamma^{N}
\cap\{(x,y) \in\RR^2 \dvtx  y \geq0 \} \subseteq
\RR\times[ 0 , 2N ] .
\]
Taken with \eqref{leftendpoint}--\eqref{rightendpoint}, we deduce that
$\Ga(N)$ contains an open path $\Gamma'$ that crosses
$B_{(\alpha-1) N, 2N}$ in the horizontal direction.

Since
$B_{(\alpha-1) N, 2N}$ lies entirely in the triangular part of $(\Sd
\comp\Tu)^N \Lattice$,
we have by Proposition~\ref{P-invariant} that
\begin{eqnarray*}
\PP_\bp^{\Lattice}[\Ch(\alpha N, N )]
&\leq&\PP_\bp^{(\Sd\comp\Tu)^N \Lattice}\bigl[\Ch\bigl((\alpha-1)
N, 2N \bigr)\bigr] \\
&=& \PP_{\bp}^\triangle\bigl[\Ch\bigl((\alpha-1) N, 2N \bigr)\bigr],
\end{eqnarray*}
and the proposition is proved.
\end{pf*}

\begin{pf*}{Proof of Proposition \protect\ref{verticalRSWsquaretotriangle}}
Consider the box $B_{N,N}$ in the mixed triangular lattice $\Lattice$
with interface-height $h(I_\Lattice) = N$. We follow the strategy of the
previous proof by considering the action of $\Sd\comp\Tu$ on a
vertical open crossing
$\Ga$ of
the box. In $N$ applications of $\Sd\comp\Tu$, the lattice
within the box is transformed
from square to triangular. By Proposition~\ref{distancebetweenpaths}(c),
the image of $\Ga$ may
drift by distance 1 or less at each step.
Drift of $\Ga$ in the horizontal direction can be accommodated
within a box, that is, wider in that direction.
Vertical drift is, however, more troublesome.
Whereas the lower endpoint of $\Ga$ is
unchanged by $N$ applications of $\Sd\comp\Tu$, its upper
endpoint may be reduced in height by $1$ at each such application.
If this were to occur at every application,
both endpoints of the final path would be on the $x$-axis.
This possibility will be controlled by proving that the downward velocity
of the upper endpoint is strictly less than $1$.

Let $\bp\in[0,1)^3$ be self-dual with $p_0>0$,
and write $\Lat^k = (\Sd\comp\Tu)^k\Lat$ for $0 \le k\le N$. The
lattice $\Lat^k$ has edge-set $E^k$ and configuration space
$\Om^k=\{0,1\}^{E^k}$. Let $\PP_\bp^k$ denote the probability
measure on $\Om^k$ given before Proposition~\ref{P-invariant}.
Recall from
that proposition that $\Sd\comp\Tu$
acts as a \textit{random} mapping from $\Om^k$ to $\Om^{k+1}$,
via the ``kernel'' given in Figure~\ref{figsimpletransformationcoupling}.
We shall assume that sequential applications of this kernel
are independent of one another and of the choice of initial configuration.
More specifically, let $(\omega^k\dvtx  k \ge0)$ satisfy the following:
\begin{longlist}[(a)]
\item[(a)]$\omega^k$ is a random configuration from $\Om^k$,
\item[(b)] the sequence $(\omega^k\dvtx  k \ge0)$ has the Markov property,
\item[(c)] given $\{\omega^0,\omega^1,\dots,\omega^k\}$,
$\omega^{k+1}$ may be expressed as $\omega^{k+1} = \Sd\comp\Tu
(\omega^k)$,
\item[(d)] the law of $\omega^0$ is $\PP_\bp^0$.
\end{longlist}
Let $\PP$ denote the joint law of the sequence $(\omega^0,\omega^1,
\dots)$.
By Proposition~\ref{P-invariant}, the law of $\omega^k$ is $\PP_\bp^k$.

Let $D^k = B_{N+k,\oo} = [-N-k,N+k] \times[0,\infty)$ be viewed
as a subgraph of~$\Lat^k$,
and call the line $\RR\times\{0\}$ the \textit{base} of $\RR^2$.
We shall work with the sequence $(h^k\dvtx  1\le k \le N)$
of random variables given by
\[
h^k := \sup\{h\dvtx  \exists x_1,x_2 \in\RR\mbox{ with } (x_1,0)
\mathop{\xleftrightarrowt}^{D^k, \omega^k}
(x_2,h)\}.
\]
Note that $h^k$ acts on $\Om^k$.

Since $\Lat^N$ is entirely triangular in the upper half-plane,
it suffices to show the existence of $\rho_N=\rho_N(\b) > 0 $ such
that $\rho_N\to1$ and
%
\begin{equation}\label{G1}
\PP( h^N \geq\b N ) \geq\rho_N \PP( h^0 \geq N ),
\end{equation}
with $\beta$ as in \eqref{G12}.
The remainder of this subsection is devoted to proving this.

\begin{lemma}\label{G-lem1}
For $0\le k < N$, the following two statements hold:
\begin{eqnarray}
h^{k+1} &\geq& h^k - 1, \label{limiteddecrease} \\
\PP\bigl(h^{k+1} \geq h+\tfrac12 \vert h^k =h\bigr) &\geq&
\b,\qquad  h \ge0.
\label{chanceofnondecrease}
\end{eqnarray}
\end{lemma}

\begin{pf}
We may assume that $h^k<\oo$ for $0 \le k \le N$, since the converse has
zero probability. Let $k<N$, and
let $\Gamma^k=\Ga^k(\omega^k)$ be the leftmost path in $D^k$ that
reaches some point at height $h^k$.
By Proposition~\ref{distancebetweenpaths}(c),
$\Lat^{k+1}$ possesses an open vertical
crossing of $B_{N+k+1,h^k-1}$, so that $h^{k+1}\ge h^k-1$. Inequality
\eqref{limiteddecrease} is proved, and we turn to \eqref
{chanceofnondecrease}.

Let $0\le k < N$, and let
$\sG$ be the set of all paths $\g$ of
$\Lat^k$ such that there exists $h>0$ with the following:
\begin{longlist}[(a)]
\item[(a)] all vertices of $\g$ lie in $B_{N+k,h}$,
\item[(b)]$\g$ has one endpoint (denoted $\g_0$) in $\RR\times\{0\}$,
\item[(c)] its other endpoint (denoted $\g_1$) lies in $\RR\times\{h\}$.
\end{longlist}
For $\g\in\sG$, there is a unique such $h$, denoted $h(\g)$.

Let $\g\in\sG$, and let $L(\g)$ be the closed subregion of
$[-N-k,N+k]\times[0,h(\g)] \subseteq\RR^2$ lying ``to the left''
of $\g$. Let $\sG(\g)$ be the subset of $\sG$ containing all paths
$\g'$ with $h(\g')=h(\g)$ and $\g'
\subseteq L(\g)$. We write $\g' < \g$ if $\g' \subseteq L(\g)$ and
$\g'\ne\g$.

\begin{figure}

\includegraphics{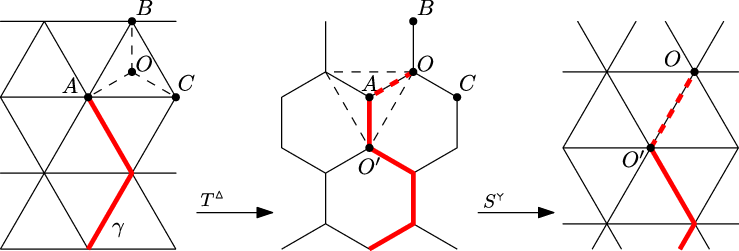}

\caption{An illustration of the action of $\Sd\circ\Tu$ when $\Ga
^k=\g$.
The top endpoint $A$ of $\g$ is preserved under $\Tu$. If $\omega^k(BC)=0$,
there is a strictly positive probability that $AO$ is open in $\Tu
(\omega^k)$,
in which case $h^{k+1} \ge h^k + \frac12$.}
\label{figchanceofnondecrease}
\end{figure}

Suppose that $p_1 \le p_2$.
The endpoint $\g_1$ is the lower \textit{left} corner of some upward
pointing triangle
denoted $ABC = ABC(\g)$, where $A = \g_1$ and $O$ is its center.
If $p_2>p_1$, we work instead with the similar triangle of which $\g
_1$ is the
lower \textit{right} corner, and the ensuing argument is exactly similar.
See Figure~\ref{figchanceofnondecrease}.

We claim that
%
\begin{equation}\label{G4}
\PP( BC \mbox{ is $\omega^k$-closed} \vert\Gamma^k =\g)
\geq1-p_1,\qquad \g\in\sG.
\end{equation}
Since the marginal of $\PP$ on $\Om^k$ is $\PP_\bp^k$, it suffices
to show that
%
\begin{equation}\label{G3}
\PP_\bp^k ( BC \mbox{ closed} \vert \Gamma^k =\g)
\geq1-p_1,\qquad \g\in\sG.
\end{equation}
This is proved as follows.
Let $\g\in\sG$. Then $\{\Ga^k = \g\} = F \cap G \cap\{\g\mbox{
open}\}$,
where~$F$ is the event that there exists no $\g'< \g$ such that every
edge of
$\g'\setminus\g$ is open,
and~$G$ is the event that there exists no $\g''\in\sG$ with $h(\g
'')>h(\g)$ and every
edge of $\g''\setminus\g$~is open. Since $F \cap G$
is a decreasing event that is independent of the states of edges in $\g$,
we have by the positive association of $\PP_\bp^k$ that
\begin{eqnarray*}
\PP_\bp^k(\Ga^k=\g\vert  BC \mbox{ closed})
&= &\PP_\bp^k(\g\mbox{ open}) \PP_\bp^k(F \cap G \vert  BC \mbox{
closed})\\
&\ge&\PP_\bp^k(\g\mbox{ open})\PP_\bp^k(F \cap G) = \PP_\bp
^k(\Ga^k=\g).
\end{eqnarray*}
Therefore,
\begin{eqnarray*}
\PP_\bp^k ( BC \mbox{ closed} \vert \Gamma^k =\g) &=&
\PP_\bp^k(\Ga^k=\g\vert  BC \mbox{ closed})
\frac{\PP_\bp^k(BC \mbox{ closed})}{\PP^k_\bp(\Ga^k = \g)}\\
&\ge&\PP_\bp^k(BC \mbox{ closed}) = 1-p_1,
\end{eqnarray*}
and \eqref{G4} is proved.

Consider the state of the edge $AO$ in the configuration $\Tu(\omega^k)$.
By Figure~\ref{figsimpletransformationcoupling},
for any $\omega\in\Om^k$ with $\omega(BC)=0$,
\[
\PP_\bp^k \bigl( AO \mbox{ open in }\Tu(\omega) | \omega
^k =\omega\bigr)
\geq\frac{p_0 p_2}{(1-p_0)(1-p_2)}.
\]
It follows that
\[
\PP\bigl(h^{k+1} \ge h^k +\tfrac12 | \omega^k=\omega\bigr)
\ge\frac{p_0 p_2}{(1-p_0)(1-p_2)} 1_{\{\omega(BC)=0\}}, \qquad\omega
\in\Om^k,
\]
where $1_H$ denotes the indicator function of an event $H$. Recall that
$BC=BC(\Ga^k(\omega))$.
Therefore, for $\g\in\sG$,
\begin{eqnarray*}
\PP\bigl(h^{k+1}\ge h^k+\tfrac12 | \Ga^k = \g\bigr)
&\ge&
\frac{p_0 p_2}{(1-p_0)(1-p_2)}\PP(\omega^k(BC) = 0 \vert \Ga^k=\g)\\
& \ge&\frac{(1-p_1)p_0 p_2}{(1-p_0)(1-p_2)},
\end{eqnarray*}
by \eqref{G4}.

Now $p_0$ is fixed, $p_1 \le p_2$,
and $\kappa_\tri(\bp)=0$. Hence,
the last ratio is a minimum when $p_1=p_2$, whence
\[
\frac{(1-p_1)p_0 p_2}{(1-p_0)(1-p_2)} \ge
\frac{1-\sqrt{1-p_0(1-p_0)}}{1-p_0} = \b,
\]
and the claim of the lemma follows.
\end{pf}

There are at least two ways to complete the proof of Proposition
\ref{verticalRSWsquaretotriangle}, of which one involves
controlling the mean of $h^{k+1} - h^k$. We take a second route
here, via a small standard lemma. For a real-valued discrete
random variable $X$, we write $\law(X)$ for its law,
and $\sS(X) := \{x \in\RR\dvtx  P(X=x)>0\}$ for its \textit{support}. The
inequality
$\lest$ denotes stochastic domination.

\begin{lemma}\label{Markovdomination}
Let $(X_0,X_1)$ and $(Y_0,Y_1)$ be pairs of real-valued discrete random
variables such that:
\begin{longlist}[(a)]
\item[(a)]$X_0 \lest Y_0$,
\item[(b)] for $x \in\sS(X_0)$, $y\in\sS(Y_0)$ with $x \le y$, the
conditional laws
of $X_1$ and $Y_1$ satisfy
$\law(X_1 \vert  X_0 = x) \lest\law(Y_1 \vert  Y_0 = y)$.
\end{longlist}
Then $X_1 \lest Y_1$.\vadjust{\goodbreak}
\end{lemma}

\begin{pf}
We include a proof for completeness. By Strassen's theorem  (see~\cite{Lindv}, Section~IV.1),
there exists a probability space
and two random variables $X_0'$, $Y_0'$, distributed respectively
as $X_0$ and $Y_0$, such that $P(X_0' \le Y_0') = 1$. Now,
\begin{eqnarray*}
P(X_1 > u) &=& \sum_{x \le y} P(X_1 > u \vert  X_0 = x) P(X_0' = x,
Y_0'=y)\\
&\le&\sum_{x \le y}P(Y_1 > u \vert  Y_0 = y) P(X_0' = x, Y_0'=y)\\
&=& P(Y_1 > u),
\end{eqnarray*}
where the summations are restricted to $x \in\sS(X_0)$ and $y\in\sS(Y_0)$.
\end{pf}

Let $(H^k\dvtx  k \ge0)$ be a Markov process with $H^0=h^0$ and transition
probabilities
%
\begin{equation}
P(H^{k+1} = j \vert  H^k = i) =
\cases{
\b,&$\mbox{if } \displaystyle j=i+\tfrac12,$ \vspace*{2pt}\cr
1- \b,&$\mbox{if } j=i-1,$}
\end{equation}
with $\b$ as above.
By Lemma~\ref{G-lem1} and an iterative application of Lemma \ref
{Markovdomination},
\[
\PP( h^N \geq\b N ) \geq
P( H^N \geq\b N) .
\]
Since $h^0$ and $H^0$ have the same distribution,
\begin{eqnarray*}
\frac{\PP(h^N \ge\b N)}{\PP(h^0 \ge N)} &\ge&
\frac{P( H^N \geq\b N)}{P(H^0 \geq N)}\\
&\ge& P(H^N \ge\b N \vert  H^0 \ge N)
=: \rho_N(\b).
\end{eqnarray*}
Now, $(H_k)$ is a random walk with mean step-size $-1+3\b/2$. By the
law of large numbers,
$\rho_N \to1$ as $N\to\oo$. In addition, $\rho_N >0$,
and \eqref{G1} follows.~%
\end{pf*}

\subsection{\texorpdfstring{Proof of Theorem \protect\ref{bxp}\textup{(b)}}
{Proof of Theorem 3.3(b)}}\label{sectprooftriangletosquare}
By Proposition~\ref{RSW}, it suffices to prove the following two propositions.

\begin{figure}

\includegraphics{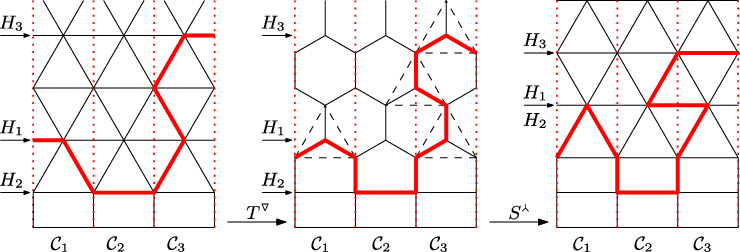}

\caption{The evolution of the heights of a crossing within columns,
when applying $\Td$
and $\Su$.
The heights in each column are the same in the first and second lattice.
In the third, $H_1$ increases by~$1$;
$H_2$ increases by $2$;
$H_3$ does not change.}
\label{fighorizontalRSWtriangletosquarepathdrift}
\end{figure}

\begin{prop} \label{horizontalRSWtriangletosquare}
Let $\bp=(p_0, p_1, p_2)\in[0,1)^3$ be self-dual with $p_0>0$.
There exists $\beta=\b(p_0) \in\NN$ and $N_0 = N_0(p_0) \in\NN$
such that,
for $\alpha\in\sqrt3 \NN$ with $\a> \beta$, and $N \ge N_0$,
\begin{equation}\label{G14}
\PP_{(p_0, 1-p_0)}^\square\bigl[\Ch\bigl( (\alpha- \beta) N, \beta N
\bigr)\bigr]
\ge(1-\a e^{-N})\PP_\bp^\triangle[\Ch(\alpha N,N )].
\end{equation}
\end{prop}

\begin{prop} \label{verticalRSWtriangletosquare}
Let $\bp=(p_0, p_1, p_2)\in[0,1)^3$ be self-dual.
For $\alpha> 0$ and $N \in2\NN$,
\begin{equation}
\PP_{(p_0, 1-p_0)}^\square\bigl[\Cv\bigl(\bigl( \alpha+ \tfrac12\bigr) N ,
\tfrac12 N \bigr)\bigr] \geq\PP_\bp^\triangle[\Cv(\alpha N, N)].
\end{equation}
\end{prop}

\begin{pf*}{Proof of Proposition \protect\ref{horizontalRSWtriangletosquare}}
Let $\bp$ satisfy the hypothesis of the proposition.
The idea is to consider repeated applications of the transformation
$\Su\comp\Td$
to an open horizontal crossing of a box in the triangular part of a
mixed lattice.
The interface moves upward, and the crossing may ``drift''
upward at each step. A new technique is required to control the rate
of this drift. This will be achieved by bounding the vertical
displacement of the path
by a certain growth process.

We partition the plane into vertical columns
\[
\col_n = \bigl(n\sqrt{3}, (n+1)\sqrt{3} \bigr) \times\RR,\qquad
n \in\ZZ,
\]
of width $\sqrt3$.
Let $\Lat=(V,E)$ be a mixed lattice, and $\omega\in\Om_E$.
The $\col_n$ correspond to the columns of the square sublattice of
$\Lat$,
as illustrated in
Figure~\ref{fighorizontalRSWtriangletosquarepathdrift}.

For any (parametrized) open path $\Lambda=(\La_t\dvtx  a\le t \le b)$ on
$\Lat$, let
\[
H_n(\Lambda) = \sup\{ h(\Lambda_t) \dvtx  t \mbox{ such that }
\Lambda_t \in\col_n \}
\]
be its \textit{height} in $\col_n$. (The supremum of
the empty set is taken to be $-\oo$.)
Note that $h(\La)= \sup_n H_n(\La)$.
The growth of the $H_n(\La)$ may be bounded as follows under the action
of the random map $\Su\comp\Td$.

For future use, we define $\eta\dvtx (0,1) \to(0,1)$ by
%
\begin{equation}\label{G13}
\eta(x) = \bigl(1+x-\sqrt{1-x+x^2}\bigr)^2,
\end{equation}
and note that $\eta$ is increasing.

\begin{lemma} \label{columngrowthcontrol}
Let $\Lat$ be a mixed triangular lattice, and let $\omega$, $\La$ be
as above.
There exists
a family of independent Bernoulli random variables $(Y_n \dvtx\break n \in\ZZ)$
with parameter $1-\eta(p_0)$,
such that, for all $n\in\ZZ$,
\[
H_{n}\bigl((\Su\comp\Td)(\Lambda)\bigr) \leq\max
\{ H_{n-1}(\Lambda),H_n(\Lambda) + Y_n , H_{n+1}(\Lambda)
\}.
\]
\end{lemma}

We delay the proof of this lemma until later in this subsection.

Let $\Lat^0=(V,E)$ be the mixed triangular lattice
with interface-height\break $h(I_{\Lat^0}) = 0$,
and let $\omega^0\in\Om_E$. Let
$\a\in\sqrt3 \NN$, and let $\Ga^0$ be an open path of $\Lat^0$ in
the box
$B_{\a N,N}$. We
shall use the notation introduced at the start of the proof of
Proposition~\ref{verticalRSWsquaretotriangle},
with the difference that the transformation $\Sd\circ\Tu$ there is
replaced here by $\Su\circ\Td$.
Thus, $\Lat^k = (\Su\circ\Td)^k\Lat$, and $\omega^k$ is the
edge-configuration on $\Lat^k$
given by $\omega^{k} = \Su\comp\Td(\omega^{k-1})$ for
$k \ge1$. Recall that $\omega^{k}$ is
a random function of $\omega^{k-1}$ generated via the kernel of Figure
\ref{figsimpletransformationcoupling},
and we assume as before that sequential applications of this kernel are
independent. We shall
study the heights of the image paths $\Ga^k = (\Su\circ\Td)^k(\Ga^0)$.

As before, if $\omega^0$ is chosen according to $\PP_\bp^0$,
then the law of $\omega^k$ is $\PP_\bp^k$.
The law of the sequence $(\omega^k\dvtx  k \ge0)$ is written $\PP$, although
for the moment we take $\omega^0$ to be fixed and write $\PP(\cdot
\vert \omega^0)$ for
the corresponding conditional measure.

We shall show that the speed of growth of the maximal height of $\Gamma^k$
is strictly less than $1$.
This will be proved by constructing
a certain growth process that dominates (stochastically)
the family $[H_n(\Ga^k)\dvtx  n \in\ZZ, k \ge0]$.

Let $\zeta\in(0,1)$. Let $(Y_n^k\dvtx  n \in\ZZ, k \ge0)$ be a family
of independent Bernoulli random variables
with parameter $1-\zeta$. The Markov process $\bX^k:=(X_n^k\dvtx\break n \in
\ZZ)$ is given as follows:
\begin{longlist}[(a)]
\item[(a)] The initial value $\bX^0$ is given by
\[
X_n^0 =
\cases{
N, &\quad$\mbox{for } n \in\bigl[- \alpha N/\sqrt3, \alpha
N/\sqrt3\bigr],$\vspace*{2pt}\cr
-\infty,&\quad$\mbox{for } n \notin\bigl[- \alpha N/\sqrt3, \alpha N/\sqrt3\bigr].$}
\]
\item[(b)] For $k \ge0$, conditional on $\bX^k$, the vector $\bX^{k+1}$
is given by
\[
X_n^{k+1} = \max\{ X_{n-1}^k, X_{n}^k + Y_n^k , X_{n+1}^k \},\qquad n
\in\ZZ.
\]
\end{longlist}

\begin{lemma}\label{growthprocessdomination}
Let $\zeta\in(0,1)$.
There exist $\b, N_0 \in\NN$ depending on $\zeta$ only (independent of
$\a$, $N$)
such that, for $\a\in\sqrt3 \NN$ and $N \geq N_0$,
\[
P \Bigl( \max_n X_n^{\beta N} \leq\beta N \Bigr) \geq1 - \alpha
e^{ - N}.
\]
\end{lemma}

We postpone the proof of this lemma,
first completing that of Proposition~\ref{horizontalRSWtriangletosquare}.
Let $\zeta=\eta(p_0)$, and let $\b$ and $N_0$ be given as in Lemma
\ref{growthprocessdomination}.
Since $H_n(\Ga^0) \le X_n^0$ for all $n$, we have by Lemma~\ref{columngrowthcontrol}
that, given $\omega^0$, $h(\Ga^k)$ is dominated stochastically by
$\max_n X_n^k$.
By Lemma~\ref{growthprocessdomination},
\begin{equation}
\PP\bigl( h(\Gamma^{\beta N}) \leq\beta N | \omega^0
\bigr) \geq1 - \alpha e^{ - N},\qquad
N \ge N_0.
\label{overallgrowthcontrol}
\end{equation}
Since $h(I_{\Lattice^0}) = 0$ and $h(I_{\Lattice^N}) = N$,
\begin{eqnarray*}
\PP_\bp^\tri[\Ch(\alpha N,N )] &= &\PP\bigl( \omega_0 \in\Ch(
\alpha N, N ) \bigr) ,\\
\PP_{(p_0, 1-p_0)}^{\square}\bigl[\Ch\bigl( (\alpha- \beta) N, \beta N \bigr)\bigr]
&=& \PP\bigl(\omega_{\beta N} \in\Ch\bigl( (\alpha- \beta) N, \beta N
\bigr) \bigr) .
\end{eqnarray*}
Hence,
\begin{eqnarray}\label{omegaconditional}
&&\frac{\PP_{(p_0, 1-p_0)}^\square[\Ch( (\alpha- \beta) N, \beta N
)]}{\PP_p^\triangle[\Ch(\alpha N,N )]}
\nonumber
\\[-8pt]
\\[-8pt]
\nonumber
&&\qquad\geq\PP\bigl(\omega_{\beta N} \in\Ch\bigl( (\alpha- \beta) N, \beta
N \bigr) |
\omega_0 \in\Ch(\alpha N, N) \bigr).
\end{eqnarray}

Let $\omega_0 \in\Ch(\alpha N, N)$ and let $\Ga^0$ be an $\omega
^0$-open crossing of $B_{\alpha N, N}$.
By Proposition~\ref{distancebetweenpaths}, the leftmost point
of $\Gamma^{\beta N}$ lies to the left of $B_{(\alpha-\beta) N,
\beta N}$, and the rightmost point
to the right of that box.
Moreover, $\Gamma^{\beta N}$ is contained in the upper half-plane, since
the lower half-plane is in the square-lattice part of every $\Lat^k$.
If, in addition, $h(\Gamma^{\beta N}) \leq\beta N$, then
$\Gamma^{\beta N}$ contains a $\omega^{\b N}$-open horizontal
crossing of $B_{(\alpha-\beta) N, \beta N}$.
In conclusion,
\begin{eqnarray*}
&&\PP\bigl(\omega^{\beta N} \in\Ch\bigl( (\alpha- \beta) N, \beta N \bigr)
|
\omega_0 \in\Ch(\alpha N, N) \bigr) \\
&&\qquad\geq\PP\bigl( h(\Gamma^{\beta N}) \leq\beta N | \omega_0
\in\Ch(\alpha N, N) \bigr) \\
&&\qquad\geq1 - \alpha e^{ - N},\qquad  N \geq N_0,
\end{eqnarray*}
by \eqref{overallgrowthcontrol}.
The claim follows by \eqref{omegaconditional}.
\end{pf*}

\begin{pf*}{Proof of Lemma \protect\ref{columngrowthcontrol}}
We recall two properties of the transformations $\Su$ and $\Td$
when applied to an $\omega$-open path $\La$.
In constructing $\Td(\La)$, we apply $\Td$ to downward pointing triangles
of $\Lat$ containing either one or two edges of $\La$.
As illustrated in Figure~\ref{figsimpletransformationcoupling},
$\Td$ acts deterministically on such triangles, and,
hence, $\Td(\La)$ is specified by knowledge of $\La$.
By inspection of Figure~\ref{fighorizontalRSWtriangletosquarepathdrift}
or otherwise,
\begin{equation}\label{TddoesnotchangeH}
H_n(\Td(\Lambda)) = H_n(\Lambda),\qquad n \in\ZZ.
\end{equation}

The situation is less simple when applying $\Su$ to $\Td(\La)$. Let
$\sS$ be
the set of upward pointing stars of $\Td\Lat$, and let $(Z^s_\rl,
Z^s_{\rm r}\dvtx  s \in\sS)$ be independent
Bernoulli random variables with parameter
\[
\nu:= \sqrt{1-\nu_0} \qquad\mbox{where } \nu_0 := 1 - \frac
{p_1p_2}{(1-p_1)(1-p_2)}.
\]
For $s \in\sS$, let $\ul Z^s = \min\{Z_\rl^s, Z_{\rm r}^s\}$,
noting that
%
\begin{equation}\label{G9}
P(\ul Z^s=1)= \nu^2 = 1-\nu_0.
\end{equation}
We call $s \in\sS$ a \textit{horizontal star} (for $\La$) if $\Td
(\La)$
includes the two nonvertical edges of $s$.

By \eqref{TddoesnotchangeH}, any changes in the $H_n$ occur only
when applying $\Su$.
The height $H_n(\Lambda)$ may grow under the application of $\Su\circ
\Td$
for either of two reasons: (i)~the highest part of $\La$ within $\col_n$ may move upward,
or (ii) part of $\La$ in a neighboring column may drift into $\col_n$ (in
which case, we say it ``invades''~$\col_n$).
These two possibilities will be considered separately.

Let $n \in\ZZ$.
Assume first that
%
\begin{equation}\label{G6}
H_{n}(\Lambda) \leq\max\{ H_{n-1}(\Lambda), H_{n+1}(\Lambda
) \} - 1.
\end{equation}
By Proposition~\ref{distancebetweenpaths},
the part of $\Lambda$ within $\col_n$ cannot drift upward by more
than 1.
By considering the ways in which parts of $\Lambda$ may invade $\col_n$,
we find that such invasions may occur only horizontally, and not
diagonally upward (see Figure \ref
{fighorizontalRSWtriangletosquarepathdrift}).
Combining these two observations, we deduce under \eqref{G6} that
\begin{equation}
H_{n}\bigl(\Su\comp\Td(\Lambda)\bigr) \leq
\max\{ H_{n-1}(\Lambda), H_{n+1}(\Lambda) \}. \label{invasiongrowth}
\end{equation}

\begin{figure}[b]

\includegraphics{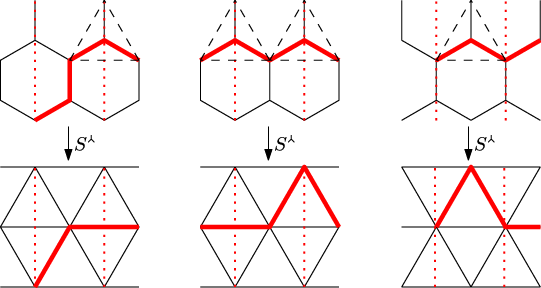}

\caption{Three examples of growth of path-height within a column
under the action of $\Su$, under the assumption
$H_{n}(\Gamma_{k}) \geq\max\{ H_{n-1}(\Gamma_{k}),H_{n+1}(\Gamma
_{k}) \}$.
\textit{Left}: The base of the marked triangle is present in the image,
and the height does not increase.
\textit{Middle}: The base of the rightmost marked triangle is absent.
The heights in the central and right columns increase.
There is a strictly positive probability that both marked bases are present,
and that the height in the central column does not increase.
\textit{Right}: The base of the marked triangle is absent, and the
height increases by $1$.}
\label{fighorizontalRSWtriangletosquarepathcolumngrowth}
\end{figure}

Suppose next that
%
\begin{equation}\label{G7}
H_{n}(\Lambda) \geq\max\{ H_{n-1}(\Lambda), H_{n+1}(\Lambda
) \}.
\end{equation}
By Proposition~\ref{distancebetweenpaths}, \eqref{TddoesnotchangeH}
and the above remark concerning invasion,
\[
H_{n}\bigl(\Su\comp\Td(\Lambda)\bigr) \leq H_{n}(\La') + 1 = H_n(\La)+1,
\]
where $\La'=\Td(\La)$. Assume that $H_{n}(\Su(\La')) = H_{n}(\La
') + 1$.
There must exist a star $s \in\sS$ such that:
\begin{longlist}[(a)]
\item[(a)]$s$ is a horizontal star for $\La$,
\item[(b)]$s$ intersects $\col_n$,
\item[(c)]$H_n(\Td(\La)) = h(O)$ where $O$ is the center of $s$,
\item[(d)] the base of $\Su(s)$ is closed in $\Su\circ\Td(\omega)$.
\end{longlist}
(See the middle and rightmost cases of Figure
\ref{fighorizontalRSWtriangletosquarepathcolumngrowth} for
illustrations.)\vadjust{\goodbreak}

Let $s$ satisfy (a), (b) and (c), and write $A$ for the highest vertex
of $s$, so that $\Td(\La)$
includes the edges $BO$ and $CO$. The edge $BC$ is open in $\Su\circ
\Td(\omega)$
with (conditional) probability
\[
\cases{
1, &\quad$\mbox{if $AO$ is closed in $\Td(\omega)$},$\vspace*{2pt}\cr
\nu_0, &\quad $\mbox{if $AO$ is open in $\Td(\omega)$}.$}
\]
(See Figure~\ref{figsimpletransformationcoupling}.)
This conditional probability is achieved by declaring $BC$ to be open
if and only if
either $AO$ is closed in $\Td(\omega)$, or $AO$ is open in $\Td
(\omega)$ and $\ul Z^s=0$.
With this coupling,
\[
\mbox{if (d) above holds, then $\ul Z^s=1$, and hence $Z^s_\rl
=Z^s_{\rm r}= 1$}.
\]

We return to \eqref{G7}.
If the highest part of $\Lambda$ in $\col_n$ comprises a single
horizontal star $s$,
as on the right of Figure \ref
{fighorizontalRSWtriangletosquarepathcolumngrowth},
\begin{equation}
H_{n}\bigl(\Su\comp\Td(\Lambda)\bigr) - H_{n}(\Lambda) \leq\max\{Z^s_\rl
, Z^s_{\rm r}\} =: Y_n. \label{verticalgrowth1}
\end{equation}
If, on the other hand, the highest part of $\Lambda$ in $\col_n$
corresponds to two stars, $s_1$ and $s_2$,
that also intersect $\col_{n-1}$ and $\col_{n+1}$, respectively (as
in the first and second diagrams of the figure),
\begin{equation}
H_{n}\bigl(\Su\comp\Td(\Lambda)\bigr) - H_{n}(\Lambda) \leq\max\{
Z^{s_1}_{\rm r}, Z^{s_2}_\rl\} =: Y_n. \label{verticalgrowth2}
\end{equation}
Recalling the properties of the $Z_\rl^s$, $Z_{\rm r}^s$, we have that the
$Y_n$ are independent Bernoulli variables with parameter $1-\eta'$, where
%
\begin{equation}\label{G11}
\eta':= \bigl(1 - \sqrt{1-\nu_0} \bigr)^2
=\Biggl (1-\sqrt{\frac{p_1p_2}{(1-p_1)(1-p_2)}}\Biggr)^2.
\end{equation}
The proof is completed by the elementary exercise of showing that $\eta
' \ge\eta(p_0)$.
 \end{pf*}

\begin{pf*}{Proof of Lemma \protect\ref{growthprocessdomination}}
The process $\bX=(\bX^k\dvtx  k \ge0)$ may be represented physically as follows.
Above each integer is a pile of bricks, illustrated in Figure~\ref{fighorizontalRSWtriangletosquaregrowthprocess}.
At each epoch of time, each column gains a random number of bricks.
If a column is at least as high as its two nearest neighboring columns,
a brick is added with probability $1-\zeta$.
Otherwise, bricks are added to the column to match the height of its
higher neighbor.

\begin{figure}

\includegraphics{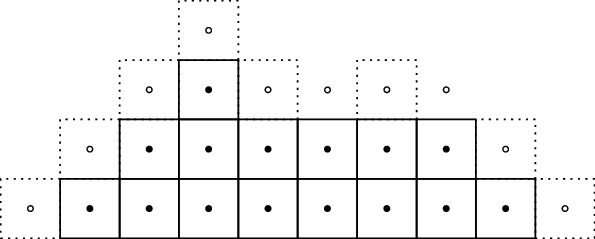}

\caption{The solid squares represent the bricks at step $k$ in the
growth process.
The dotted squares are the additions at time $k+1$.
The lateral extensions occur with probability $1$, and the
vertical extensions with probability $1-\zeta$.}
\label{fighorizontalRSWtriangletosquaregrowthprocess}
\end{figure}

We study the process via the times at which bricks are placed at vertices.
For each pair $A$, $B$ of neighbors in the upper half-plane
$\ZZ\times\ZZ_0$ of the square lattice with the usual embedding, we
place a directed edge
denoted $AB$
from $A$ to $B$, and similarly a directed edge $BA$ from $B$ to $A$.
Let $\sE$ be
the set of all such directed edges.
The random variables $(\tau_{AB}\dvtx  AB \in\sE)$ are assumed independent
with distributions as follows:
\[
\tau_{AB} =
\cases{
1, &\quad$\mbox{if $AB$ is horizontal},$\vspace*{2pt}\cr
0, &\quad$\mbox{if $AB$ is directed downward},$}
\]
and $\tau_{AB}$ has the geometric distribution with parameter $1-\zeta
$ if
$AB$ is directed upward, that is,
\[
P(\tau_{AB} = r) = \zeta^{r-1}(1-\zeta),\qquad  r \ge1.
\]
Thinking about $\tau_{AB}$ as the time for the process to pass along the
edge $AB$,
we define the \textit{passage-time} from $C$ to $D$ by
\[
\tau(C,D) = \inf\biggl\{\tau(\vec\La) := \sum_{ e\in\vec\La}
\tau_e\dvtx  \vec\La\in\sP_{C,D}\biggr\},
\]
where $\sP_{C,D}$ is the set of all directed paths from $C$ to $D$.

Let $\a\in\sqrt3\NN$ and $L_i := [-\a N/\sqrt3, \a N/\sqrt
3]\times\{i\}$.
The initial state $G_0$ of this growth process
is the set $\bigcup_{i=0}^N L_i$. It is easily seen that the
state $G_k$ at time $k$ comprises exactly the set of all vertices $D$ such
that there exists $C \in L_N$ with $\tau(C,D) \le k$.

Let $\beta> 3$ be an integer, to be chosen later. By the above,
\begin{equation}\label{G5}
P \bigl( h(G_{\beta N}) \geq\beta N \bigr)
\leq
\mathop{\sum_{C,D\dvtx }}_{C\in L_N, h(D)=\beta N} P \bigl(\tau(C,D) \leq
\beta N \bigr).
\end{equation}
Now, $\tau(C,D) \leq\beta N$ if and only if there exists
a directed path $\vec\La\in\sP_{C,D}$ with passage-time not
exceeding $\beta N$,
so that
\begin{equation}\label{boundbysum}
P \bigl( h(G_{\beta N}) \geq\beta N \bigr)
\leq
\sum_{\vec\La\in\sP_N} P\bigl(\tau(\vec\La) \le\b N\bigr),
\end{equation}
where $\sP_N$ is the set of directed paths whose endpoints $C$, $D$
are as
in \eqref{G5}.
Consider such a path $\vec\La$, and let $u$, $d$, $h$ be the numbers
of its upward, downward and horizontal edges, respectively.
Since upward and horizontal edges have passage-times at least $1$, we
must have
$u+h \le\beta N$. By considering the heights of the first and last vertices,
$u-d=(\beta-1) N$. Therefore, $\vec\La$ has no more than $(\b+1)N$ edges
in total, of which
at least $(\beta-1)N$ are upward.

There are $|L_N| \le2 \alpha N$ possible choices for $C$,
so that
\begin{equation}
|\sP_N| \le2 \alpha N 4^{2N} \pmatrix{(\beta+1)N\vspace*{2pt}\cr {2N}}. \label{boundnumberterms}
\end{equation}
For $\vec\La\in\sP_N$, $\tau(\vec\La)$ is no smaller than
the sum of the passage-times of its upward edges. Therefore,
\begin{equation}
P\bigl(\tau(\vec\La) \leq\beta N \bigr)
\leq P (S \leq\beta N),
\end{equation}
where $S$ is the sum of $(\b-1)N$ independent random variables with
the Geom$(1-\zeta)$
distribution. It is elementary that
\[
P(S \le\beta N) = P\bigl(T \ge(\b-1)N\bigr),
\]
where $T$ has the binomial distribution bin$(\b N, 1-\zeta)$. By
Markov's inequality (as in the proof of Cram\'er's theorem),
%
\begin{equation}\label{boundindividualtime}
\limsup_{N\to\oo}P\bigl(T \ge(\b-1)N\bigr)^{1/N}
\le\b\biggl(\frac{\b(1-\zeta)}{\beta-1}\biggr)^\b,
\end{equation}
when $\b(1-\zeta) < \b-1$, that is, $\b> 1/\zeta$.

By \eqref{boundbysum}--\eqref{boundindividualtime}, there exists
$N_0=N_0(\b,\zeta)$ such that, for $N \ge N_0$,
\[
P \bigl( h(G_{\beta N}) \geq\beta N \bigr)
\leq2 \alpha N 4^{2N} \pmatrix{(\beta+ 1)N\vspace*{2pt}\cr{2N}}
\biggl\{ 2\b\biggl(\frac{\b(1-\zeta)}{\beta-1}\biggr)^\b\biggr\}^N.
\]
By Stirling's formula, there exists $c=c(\zeta)$ and
$N_1=N_1(\b,\zeta)$ such that, for \mbox{$N \ge N_1$,}
\begin{equation}
P \bigl( h(G_{\b N}) \geq\b N \bigr)
\leq\a\biggl\{c\b^3\biggl(\frac{\b(1-\zeta)}{\beta-1}
\biggr)^\b\biggr\}^N .
\label{finaldominationofX}
\end{equation}
Choose $\b=\b(\zeta)$ sufficiently large that the last term is
smaller than $\a e^{-N}$,
and the proof is complete.
\end{pf*}

This concludes the proof of Lemma~\ref{growthprocessdomination}
and thus of Proposition~\ref{horizontalRSWtriangletosquare}.

\begin{pf*}{Proof of Proposition \protect\ref{verticalRSWtriangletosquare}}
Let $N \in2 \NN$.
Let $\Lattice=(V,E)$ be the mixed triangular lattice with
interface-height $0$, so that
\[
\PP_{\bp}^\triangle[ \Cv(\alpha N, N) ] =
\PP_{\bp}^\Lattice[ \Cv(\alpha N, N) ] .
\]
Let $\omega\in\Om_E$, and let $\Ga$ be an $\omega$-open vertical
crossing of $B_{\a N,N}$.
In $\frac12 N$ applications of $\Su\comp\Td$, the images of the
lower endpoint of $\Ga$ remain
in the square part of the lattice, and thus are immobile.
By Proposition~\ref{distancebetweenpaths},
$(\Su\circ\Td)^{N/2}(\Ga)$ contains\vadjust{\goodbreak} a vertical crossing of
$B_{(\alpha+ 1/2) N, N/2}$,
that is, open in $(\Su\circ\Td)^{N/2}(\omega)$.
Since $B_{(\alpha+ 1/2) N, N/2}$ lies entirely within the square
part of $(\Su\comp\Td)^{N/2} \Lat$,
we deduce that
\begin{eqnarray*}
\PP_{(p_0,1-p_0)}^\square\bigl[ \Cv\bigl(\bigl( \alpha+ \tfrac12\bigr) N, \tfrac
12 N \bigr) \bigr]
&=& \PP_{\bp}^{(\Su\comp\Td)^{N/2} \Lattice} \bigl[\Cv\bigl(\bigl(\a
+\tfrac12\bigr)N,N\bigr)\bigr] \\
&\geq&\PP_{\bp}^\triangle[ \Cv(\alpha N, N) ],
\end{eqnarray*}
and the claim is proved.
\end{pf*}

\section{Remaining proofs}\label{sectgeneralizations}

\subsection{Using the box-crossing property}\label{sectcriticalityusingRSW}

Our target in this subsection is to summarize how certain properties of
inhomogeneous percolation models may be deduced from the box-crossing property.
These properties will be used later in this section, and are of
independent interest.

We shall consider bond percolation on the square, triangular and
hexagonal lattices,
and any reference to a lattice shall mean one of these three, duly
embedded in $\RR^2$
as described after Definition~\ref{defbxp}.
Fix a lattice $\Lat= (V, E)$ with \textit{origin} $0$,
and let $\bp= (p_e\dvtx  e \in E) \in[0,1]^E$.
Denote by $\PP_\bp$ the product measure on $\Omega= \{0,1\}^E$
under which edge $e\in E$ is open with probability~$p_e$. The lattice
$\Lat$
has a dual lattice $\Lat^*=(V^*,E^*)$, and we write $\PP_{1-\bp}^*$
for the
product measure on $\Om^*=\{0,1\}^{E^*}$ under which an edge $e^* \in
E^*$, dual to $e \in E$,
is open with probability $1-p_e$.

For $\nu> 0$, let $\ol\PP_{\bp+\nu}$ be the product measure on $E$
under which $e$ is open with probability $1_{\{p_e>0\}}\min\{p_e+\nu
,1\}$,
and $\ul\PP_{\bp-\nu}$ that under which $e$ is open with probability
$1_{\{p_e=1\}} + 1_{\{p_e<1\}}\max\{p_e-\nu,0\}$. Under these measures,
any edge with $\PP_\bp$-parameter equal to $0$ or $1$ retains this property.
Let $|\cdot|$ denote the Euclidean norm as before and,
for $x \ge0$, define the box
$
S_x=\{z\in\RR^2 \dvtx  |z| < x\}.
$
The open cluster at vertex $v$ is denoted $C_v$, and its \textit{radius}
$\rad(C_v)$ is the supremum of all $x$ such that $C_v$ intersects
$\RR^2 \setminus(v+S_x)$.

\begin{prop}\label{sub-criticalityviaRSW}
Suppose $\PP_{1-\bp}^*$ has the box-crossing property.
\begin{longlist}[(a)]
\item[(a)]
There exist $a,b > 0$ such that, for every $v\in V$,
\[
\PP_{\bp} \bigl(\rad(C_v) \geq k\bigr) \leq ak^{-b},\qquad k \ge0.
\]
\item[(b)]
There exists, $\PP_{\p}$-a.s., no infinite open cluster.
\item[(c)]
For $\nu>0$, there exist $c,d>0$ such that, for every $v\in V$,
\[
\ul\PP_{\bp-\nu}(|C_v|\ge k) \le c e^{-dk},\qquad k \ge0.
\]
\end{longlist}
\end{prop}

\begin{prop}\label{propsuper}
Suppose $\PP_\bp$ has the box-crossing property.
\begin{longlist}[(a)]
\item[(a)]
There exist $a,b > 0$ and $M\in\NN$ such that for every $v\in V$,
there exists $w=w(v)\in V$ with $|v-w|\le M$ and
\[ 
\PP_{\bp} \bigl(\rad(C_w) \geq k\bigr) \ge ak^{-b},\qquad  k \ge0.
\]
\item[(b)]
Let $\nu> 0$. There exist $\a>0$ and $M \in\NN$ such that for every
$v \in V$, there
exists $w=w(v)\in V$ with $|v-w|\le M$ and
$\ol\PP_{\p+ \nu}(w \lra\oo)>\a$.
There exists, $\ol\PP_{\bp+\nu}$-a.s., a unique infinite open cluster.
\end{longlist}
\end{prop}

The unusually complicated formulation of this proposition arises from the
possible existence of edges $e$ with
$p_e=0$.

\begin{obs}\label{rem1}
Parts (a) and (b) of Propositions \eqref
{sub-criticalityviaRSW} and \eqref{propsuper}
are valid for any primal/dual pair $G/G^*$ embedded
in $\RR^2$ such that: (i) $G$ is connected, (ii) $G$ and $G^*$ are locally
finite, in that any bounded subset of $\RR^2$ contains finitely many
vertices, and (iii)
there exists $L<\oo$ such that no edge of $G$ or $G^*$ has length
exceeding~$L$.
[Condition (iii) is a consequence of the box-crossing property.]
For Proposition~\ref{sub-criticalityviaRSW}(c), condition (ii)
is replaced by the stronger (ii$'$): there exists $\rho<\oo$ such that
the number of vertices within any translate of the unit disk is no
greater than $\rho$.
The proofs follow those of the lattice case, and are omitted.
\end{obs}

\begin{pf*}{Sketch proof of Proposition \protect\ref{sub-criticalityviaRSW}}
Further details of the arguments used here may be found in
\cite{GrimmettGraphs}, Chapter~5. For definiteness, we consider only
the square lattice, and
analogous proofs are valid for the triangular and hexagonal lattices.
Assume $\PP_{1-\bp}^*$ has the box-crossing property.

(a) If $\rad(C_v) \ge k$, there exist order $\log k$ disjoint annuli
around $v$, none of which contains a dual open cycle surrounding $v$.
By the box-crossing property,
this event has probability
less than $a(1-\g)^{\a\log k}$ for some \mbox{$a,\g,\a>0$}, and the claim
follows. Part (b) is a trivial consequence.

(c)
Let $N \ge1$.
Let $B_N$ be a box of size $3N \times N$, and let $H_N$ be the event
that~$B_N$ has
an open dual crossing in the long direction, in the dual lattice $\Lat
^*$. By the box-crossing property,
there exists $\tau=\tau(\bp)>0$, independent of $B_N$ and $N$, such that
$\PP^*_{1-\bp}(H_N) \ge\tau$
for all large $N$.
By part (a) above, Russo's formula and the theory of
influence (as in~\cite{GG2} and~\cite{GrimmettGraphs}, Theorem~4.33,
e.g.),
%
\begin{equation}\label{G16}
\frac{d}{d\eta}\ol\PP^*_{1-\bp+\eta}(H_N) \ge c_1\ol\PP
^*_{1-\bp+\eta}(H_N)
\bigl(1-\ol\PP^*_{1-\bp+\eta}(H_N)\bigr)\log(c_2 N^b),
\end{equation}
for $\eta>0$ and some absolute constants $c_1, c_2 >0$. This
inequality, when integrated
over $(0,\nu)$, yields
$\ol\PP^*_{1-\bp+\nu}(H_N) \to1$ as $N \to\oo$, uniformly in the
choice of $B_N$.

For $\zeta>0$, we may choose $N$ sufficiently large that $\ol\PP
_{1-\bp+\nu}^*(H_N) \ge1-\zeta$ for
all~$B_N$. By passing to the dual and using the method of proof of
\cite{Kes81}, Theorem~1, with $\zeta$ small (see also~\cite{GrimmettRCM}, Theorem 5.86), one obtains
the exponential decay of cluster-volume.
\end{pf*}

\begin{pf*}{Sketch proof of Proposition \protect\ref{propsuper}}
For simplicity, we consider only the square lattice.

(a)
Let $v \in V$.
For odd $i \ge1$, let $A_{i}(v)$ be the event that $v + [0,2^i]\times
[0,2^{i-1}]$ has a
horizontal open crossing; for even $i \ge1$, let $A_{i}(v)$
be the event that $v +[0,2^{i-1}]\times[0,2^{i}]$ has
a vertical open crossing. By the box-crossing property, there exist
$\tau>0$ and $I
\in\NN$ such that
$\PP_\bp(A_i(v))>\tau$ for $i \ge I$ and $v \in V$.
Let $M=2^{I-1}$ and $J > I + \log_2 k$. There exists $w$ with
$|v-w|\le M$ such that on
the event $\bigcap_{i=I}^J A_i(v)$, we have $\rad(C_w) \ge k$.
The claim follows by positive association and the box-crossing property.

\begin{figure}[b]

\includegraphics{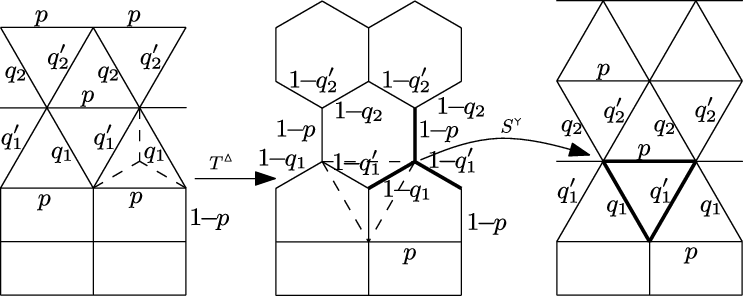}

\caption{A mixed triangular lattice (left) with the highly
inhomogeneous measure above the interface. The transformation
$\Sd\circ\Tu$ moves the interface down by one unit. Every triangle
is parametrized by a self-dual triplet.}
\label{figgeneraltriangularRSWtransfo}
\end{figure}

(b)
Of the ways of proving this, we choose
to follow the proof of Proposition~\ref{sub-criticalityviaRSW}(c)
and~\cite{GrimmettGraphs}, Theorem 5.64.
Let $N \ge1$, let $B_N=[0,8N] \times[0,2N]$, and let $f\dvtx \RR^2 \to
\RR^2$ comprise
a rotation and a translation.
Let $H_N$ be the event that $f(B_N)$ has
an open ``horizontal'' crossing, and in addition the boxes
$f([0,2N]\times[0,2N])$ and
$f([6N, 8N]\times[0,2N])$ have open ``vertical'' crossings
(``horizontal''
and ``vertical''
refer to the orientation of $B_N$). By the box-crossing property and
positive association,
there exists $\tau>0$ such that $\PP_\bp(H_N) \ge\tau$ for all
large $N$, uniformly in~$f$.
By the argument leading to \eqref{G16},
$\ol\PP_{\bp+\nu}(H_N) \to1$ as $N \to\oo$, uniformly in~$f$.
We pick $N$ sufficiently large, and adapt the block
argument of~\cite{GrimmettGraphs}, Section~5.8 to deduce the claim.
The second assertion is an elementary consequence of the box-crossing
property and the
existence of
open paths in annuli.
\end{pf*}

\subsection{\texorpdfstring{Proofs of Theorems \protect\ref{generaltriangularRSW} and \protect\ref{generaltriangularcriticality}}
{Proofs of Theorems 1.4 and 1.6}}

The proof of
Theorem~\ref{generaltriangularRSW} follows exactly that of
Section~\ref{sectproofsquaretotriangle} on noting that each
triangle of the mixed triangular lattice
of Figure~\ref{figgeneraltriangularRSWtransfo} has three edges
with parameters forming a self-dual triplet,
and the constants of Propositions
\ref{horizontalRSWsquaretotriangle} to \ref
{verticalRSWtriangletosquare}
depend only (in the current setting) on the value of $p$ and not
otherwise on $\q$ and $\q'$.
The hexagonal-lattice case follows by a single application of the star--triangle
transformation.

Since $\PP^\triangle_{p, \q, \q'}$ is increasing in $\q$ and $\q'$,
and since the nonexistence of an infinite component is a decreasing event,
Theorem~\ref{generaltriangularcriticality}(a) follows from
Proposition~\ref{sub-criticalityviaRSW}(b).

Turning to part (b) of Theorem~\ref{generaltriangularcriticality},
assume \eqref{G17} holds with $\de>0$. Let $\eps=\frac14 \de$
and note from \eqref{G17} that
$p,q_n,q_n' < 1-\eps$ for $n \in\ZZ$.
Therefore, $p+\eps,\break q_n + \eps, q'_n+\eps<1$ for all $n$, and
\[
\kappa_\tri(p + \eps,q_n + \eps, q_n' + \eps) \leq0,\qquad  n \in
\ZZ.
\]
By Theorem~\ref{generaltriangularRSW}
and the monotonicity of measures, the measure of the dual process,
$\PP^\hex_{1- p - \eps,1- \q- \eps,1- \q'- \eps}$, has the
box-crossing property.
The claim
follows by Proposition~\ref{sub-criticalityviaRSW}(c) with $\nu=
\eps$.

Assume finally that
\eqref{generaltriangularsuper-criticalitycondition} holds
with $\de>0$.
Let $\eps=\frac13 \min\{\de, p\}$ and write
\[
x^+ = \max\{x, 0\},\qquad
\what x = x1_{\{x \ge\eps\}}.
\]
Then
\[
\kappa_\tri\bigl( (p - \eps)^+,(q_n - \eps)^+, (q_n' - \eps
)^+\bigr) \geq0,\qquad  n \in\ZZ.
\]
By Theorem~\ref{generaltriangularRSW} and the monotonicity of
measures, the associated
product measure on the triangular lattice has the box-crossing property.
By Proposition~\ref{propsuper}(b) with \mbox{$\nu= \eps$},
and the fact that $\PP_{\what p, \what\q, \what\q'}((v+S_M)
\subseteq
C_v )$ is bounded from $0$ uniformly
in \mbox{$v\in V$}, we have that
$\PP_{\what{p},\what\q,\what\q'}$ is uniformly supercritical.
By monotonicity of measures, $\PP_{p,\q,\q'}$ is uniformly
supercritical as claimed.

The same arguments are valid for the hexagonal lattice.

\subsection{\texorpdfstring{Proofs of Theorems \protect\ref{generalsquareRSW} and \protect\ref{generalsquarecriticality}}
{Proofs of Theorems 1.5 and 1.7}}

Let $\q=1-\q'$ satisfy \eqref{criticalsquare2} with \mbox{$\eps>0$}, and
let $p= 1-p'=\frac12 \eps$. We may pick
$r_n\in(0,1)$ such that $\kappa_\tri(p,q_n,r_n)=0$ for all $n$,
and we write $r_n'=1-r_n$.
By Theorem~\ref{generaltriangularRSW},
the measure $\PP_{p,\q,\r}^\tri$ has the box-crossing property, and
we propose to
transport this property to
the square-lattice measure $\PP_{\q,\q'}$ via the star-triangle
transformation.

\begin{figure}

\includegraphics{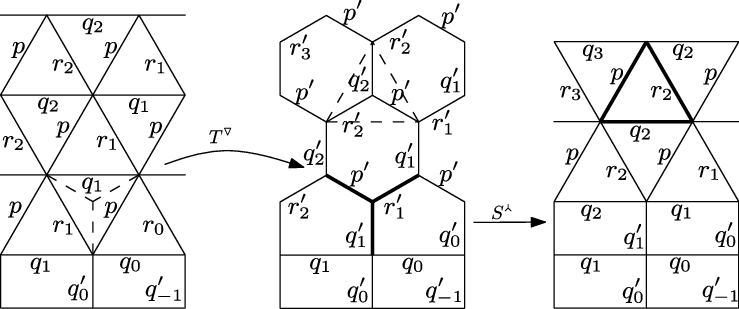}

\caption{\textit{Left}: The measure $\Pqsq$ on $\Lat$.
In the triangular part the measure is $\PP^\triangle_{p, \q, \r}$
on a rotated lattice,
and in the square part it is $\PP^\square_{\q,\q'}$.
\textit{Middle, right}: Application of $\Su\comp\Td$ transforms
$\Lat$ to a~
copy of itself shifted upward and sideways.}
\label{figgeneralsquareRSWtransfo}
\end{figure}

Let $\Lat=(V,E)$ be the mixed triangular lattice on the left
of Figure~\ref{figgeneralsquareRSWtransfo},
and denote by $\Pqsq$ the product measure given there.
Under $\Pqsq$, all triangles in $\Lattice$ have self-dual triplets.
Thus, $\Td$ acts on $\Om_E$ endowed with $\Pqsq$ in the manner of
Section~\ref{secttransformation}
(with parameters varying between triangles),
and the ensuing measure is given in the middle figure.
Then $\Su$ acts on edge-configurations of $\Td\Lattice$ (with
parameters varying between stars).
The ensuing lattice $(\Su\comp\Td)\Lattice$ is illustrated on the
right, and
it may be noted that the corresponding measure is precisely that of
$\Lat$ shifted
upward and rightward.

In the triangular part of $\Lattice$, $\Pqsq$ corresponds to the
measure $\PP^\tri_{p, \q, \r}$,
while in the square part it corresponds to $\PP^{\square}_{\q, \q'}$.\vspace*{1pt}
By Theorem~\ref{generaltriangularRSW}, $\PP^\tri_{p, \q, \r}$
has the box-crossing property,
and thus it remains to adapt the proofs of Propositions~\ref
{horizontalRSWtriangletosquare} and~\ref{verticalRSWtriangletosquare}.

Proposition~\ref{verticalRSWtriangletosquare} holds because of
its nonprobabilistic bound for the drift of a path\vadjust{\goodbreak} under $\Su\comp\Td$.
Its proof is easily adapted to give, as there, that, for $\alpha> 0$
and $N \in2\NN$,
\[
\PP_{\q, \q'}^\square\bigl[\Cv\bigl( \bigl(\alpha+ \tfrac12 \bigr) N , \tfrac
12 N \bigr) \bigr]
\geq\PP_{\q,\r,p}^\triangle[\Cv( \alpha N, N)]. %
\]

The proof of Proposition~\ref{horizontalRSWtriangletosquare}
requires the probabilistic estimate of Lem\-ma~\ref
{columngrowthcontrol}. This hinges on
the application of $\Su$ to configurations on upward pointing stars.
The key fact is
that $\eta(p_0) >0$, with $\eta$ as in \eqref{G13} and $p_0$ the parameter
associated with a horizontal edge in the triangular lattice. In the
present situation, such
edges have parameters $q_n$. Since $q_n \ge\eps$, we have that $\eta
(q_n) \ge\eta(\eps)>0$.
This results in an altered version of Lemma \ref
{columngrowthcontrol} with $\eta(p_0)$
replaced by $\eta(\eps)$. The proof continues as before, and a
version of \eqref{G14} results.
Theorem~\ref{generalsquareRSW} is proved.

Finally, consider Theorem~\ref{generalsquarecriticality}, and assume
\eqref{G15}.
Let $\nu_n=(1-q_n-q_n')/2$, and
apply Theorem~\ref{generalsquareRSW} to the self-dual measure $\PP
^{\square}_{\q+\bnu,\q'+\bnu}$.\vspace*{1pt}
Part (a) then follows by Proposition~\ref{sub-criticalityviaRSW}(b).
The proofs of (b, c)
hold as for the triangular lattice.

\section*{Acknowledgment}
G. R. Grimmett acknowledges the suggestion of Harry Kesten in 1987 that the Baxter--Enting
paper~\cite{Baxter399} might have implications for percolation.



\printaddresses

\end{document}